\documentclass[12pt]{amsart}

\usepackage{amsmath,latexsym,amsfonts,amssymb,amsthm}
\usepackage{geometry}
\usepackage{hyperref}
\usepackage{mathrsfs}
\usepackage{mathtools}
\usepackage[all,cmtip]{xy}

\newtheorem{prop}{Proposition}[section]
\newtheorem{thm}[prop]{Theorem}
\newtheorem{cor}[prop]{Corollary}

\newtheorem{lem}[prop]{Lemma}

\theoremstyle{definition}

\newtheorem{defn}[prop]{Definition}
\newtheorem{expl}[prop]{Example}
\newtheorem{rem}[prop]{\it Remark}

\newtheorem{ass}[prop]{Assumption}

\newtheorem*{claim*}{Claim}

\newcommand{\bP}{\mathbb{P}}
\newcommand{\bC}{\mathbb{C}}
\newcommand{\bR}{\mathbb{R}}

\newcommand{\bQ}{\mathbb{Q}}
\newcommand{\bZ}{\mathbb{Z}}

\newcommand{\tX}{\widetilde{X}}
\newcommand{\tY}{\widetilde{Y}}
\newcommand{\tf}{\widetilde{f}}
\newcommand{\tF}{\widetilde{F}}
\newcommand{\tD}{\widetilde{D}}

\newcommand{\cX}{\mathcal{X}}

\newcommand{\cO}{\mathcal{O}}

\newcommand{\cI}{\mathcal{I}}

\newcommand{\cJ}{\mathcal{J}}

\newcommand{\Supp}{\mathrm{Supp}~}

\newcommand{\mult}{\mathrm{mult}}
\newcommand{\Ex}{\mathrm{Ex}}
\newcommand{\Pic}{\mathrm{Pic}}
\newcommand{\vol}{\mathrm{vol}}

\begin{document}

\title{Fano varieties with large Seshadri constants in positive characteristic}
\author{Ziquan Zhuang}
% \institute{Ziquan Zhuang \at
%               Department of Mathematics, Princeton University, Princeton, NJ, 08544-1000. \\
%               \email{zzhuang@math.princeton.edu}
% }
\email{zzhuang@math.princeton.edu}
\address{Department of Mathematics, Princeton University, Princeton, NJ, 08544-1000.}
\date{}

\maketitle

\begin{abstract}
We prove that a Fano variety (with arbitrary singularities) of dimension $n$ in positive characteristic is isomorphic to $\mathbb{P}^n$ if the Seshadri constant of the anti-canonical divisor at some smooth point is greater than $n$ and classify Fano varieties whose anti-canonical divisors have Seshadri constants $n$. In characteristic $p>5$ and dimension $3$, we also show that Fano varieties $X$ with Seshadri constants $\epsilon(-K_X,x)>2+\epsilon$ at some smooth point $x\in X$ (for some fixed $\epsilon>0$) have bounded anti-canonical degrees.
% \keywords{Fano variety \and Seshadri constant \and positive characteristic \and boundedness \and singularity \and classification.}
% \subclass{14J45 (primary), 14G17, 14E99, 14C20 (secondary)}
\end{abstract}

\section{Introduction}

Let $X$ be a normal projective variety and $L$ an ample $\bQ$-Cartier divisor on $X$. The Seshadri constants of $L$, originally introduced by Demailly \cite{demailly}, serve as a measure of the local positivity of the divisor $L$.

\begin{defn}
Let $L$ be an ample $\bQ$-Cartier divisor on a projective variety $X$ and $x\in X$ a smooth point. The \textit{Seshadri constant} of $L$ at $x$ is defined as
\[
  \epsilon(L,x):=\sup\{t\in\bR_{>0}\mid \sigma^*L-tE\textrm{ is ample}\},
\]
where $\sigma:\mathrm{Bl}_x X\to X$ is the blow-up of $X$ at $x$, and $E$ is the exceptional divisor of $\sigma$.
\end{defn}

When $X$ is Fano, i.e. $-K_X$ is $\bQ$-Cartier and ample, it is natural to look at the Seshadri constants of the anti-canonical divisor. It has been observed that over a field of characteristic zero, this local invariant, when it's large, also governs the global geometry of the Fano variety. More precisely, Fano varieties with large Seshadri constants enjoy nice geometric properties and satisfy certain boundedness:

\begin{thm} \label{thm:char 0} \cite{bs,lz16,me}
Let $X$ be a complex Fano variety of dimension $n$ and let $x\in X$ be a smooth point. Let $\epsilon>0$ be a constant.
    \begin{enumerate}
        \item If $\epsilon(-K_X,x)>n$, then $X\cong\bP^n$;
        \item If $\epsilon(-K_X,x)\ge n$, then $X$ has klt singularities;
        \item If $\epsilon(-K_X,x)>n-1$, then $X$ is rationally connected;
        \item The set of those $X$ with $\epsilon(-K_X,x)>n-1+\epsilon$ for some $x\in X$ is weakly bounded;
        \item The set of those $X$ with $\epsilon(-K_X,x)\ge n-1$ for some $x\in X$ is birationally bounded.
    \end{enumerate}
Moreover, all the assumptions on the lower bound of Seshadri constants here are sharp.
\end{thm}

Here a set of Fano varieties is said to be weakly bounded if there exists a constant $M>0$ such that for any Fano variety $X$ in the set we have $\vol(-K_X)=((-K_X)^n)<M$. It is said to be birationally bounded if there exists a family of varieties $\pi:\cX\rightarrow B$ over a base of finite type such that any $X$ in the set is birational to $\cX_b=\pi^{-1}(b)$ for some $b\in B$.

The proof of this theorem relies on various consequences of the Kawamata-Viehweg vanishing theorem (e.g. basepoint-free theorem, minimal model program and Koll\'ar-Shokurov connectedness), so does not extend to positive characteristic. On the other hand, we believe that the same statements also hold over a field of characteristic $p>0$ (possibly excluding a few small primes $p$) and indeed, using Frobenius technique, Murayama \cite[Theorem B]{takumi} recently showed that if $X$ is a smooth Fano variety (over an algebraically closed field of any characteristic) of dimension $n$ and $\epsilon(-K_X,x)\ge n+1$ for some $x\in X$ then $X\cong\bP^n$, giving a partial generalization of the first statement in the above theorem.

Unless otherwise specified, all varieties in what follows are defined over an algebraically closed field $k$ of characteristic $p>0$. The goal of this paper is to provide an argument that generalizes parts of Theorem \ref{thm:char 0} to positive characteristic. Here is our first main result:

\begin{thm} \label{thm:Pn}
Let $X$ be a normal projective variety of dimension $n$ and $\Delta$ an effective $\bQ$-divisor on $X$ such that $L=-(K_{X}+\Delta)$ is $\mathbb{Q}$-Cartier and ample. Let $x\in X$ be a smooth point.  
\begin{enumerate}
    \item If $\epsilon(L,x)>n$, then $X\cong\mathbb{P}^{n}$;
    \item If $\epsilon(L,x)\ge n$ and $p>2$, then $X$ is globally $F$-regular \emph{(}see \S \ref{sec:F-sing} for definition\emph{)}.
\end{enumerate}
\end{thm}

By standard reduction mod $p$ technique and the main result of \cite{ss}, we immediately obtain a different proof of its characteristic zero analog:

\begin{cor}
Let $X$ be a normal projective variety of dimension
$n$ over $\bC$ and $\Delta$ an effective $\bQ$-divisor on $X$ such that $L=-(K_{X}+\Delta)$
is $\mathbb{Q}$-Cartier and ample. Assume that $\epsilon(L,x)\ge n$
for some smooth point $x\in X$, then $X$ is of Fano type. If in addition $\epsilon(L,x)>n$, then $X\cong\mathbb{P}^{n}$.
\end{cor}

The argument we introduce here has the additional bonus that it generalizes \cite[Theorem 3]{lz16}  (which classifies complex Fano varieties $X$ with $\epsilon(-K_X,x)=n$) to positive characteristic as well.

\begin{thm} \label{thm:=n}
Let $X$ be a normal projective variety of dimension $n$ and $\Delta$ an effective $\bQ$-divisor on $X$ such that $L=-(K_{X}+\Delta)$ is $\mathbb{Q}$-Cartier and ample. Assume that $\epsilon(L,x)=n$ for some smooth point $x\in X$ and that either $(L^n)>n^n$ or $\Delta\neq0$. Then either  $X\cong\bP^n$ or $X$ is one of the following:
\begin{enumerate}
\item a degree $d+1$ weighted hypersurface $(x_0x_{n+1}=f(x_1,\cdots,x_n))\subset\mathbb{P}(1^{n+1},d)$;
\item the blow-up of $\bP^n$ along a hypersurface of a hyperplane;
\item a Gorenstein log Del Pezzo surface of degree $\ge 5$.
\end{enumerate}
\end{thm}

Note that the condition on Seshadri constant $\epsilon(L,x)$ $=n$ already implies $(L^n)\ge n^n$. When equality holds, we have (by the above theorem, we may assume $\Delta=0$):

\begin{thm} \label{thm:n^n}
Let $X$ be a normal projective variety of dimension $n$ such that $-K_X$ is $\mathbb{Q}$-Cartier and ample. Assume that $\epsilon(-K_X,x)=n$ for some smooth point $x\in X$, $((-K_X)^n)=n^n$ and $p\neq 2$, then $X$ is one of the following:
\begin{enumerate}
\item a quartic weighted hypersurface $X_4=(x_{n+1}^2+x_n h(x_0,\cdots,x_{n-1})=f(x_0,\cdots,x_{n-1}))$ $(h\neq0)$ or $(x_n x_{n+1}=f(x_0,\cdots,x_{n-1}))\subseteq\mathbb{P}(1^n,2^2)$;
\item the quotient of the quadric $Q_k=(\sum_{i=0}^k x_i^2=0)\subseteq\mathbb{P}^{n+1}\,(2\leq k\leq n+1)$ by an involution $\tau(x_i)=\delta_i x_i\,(\delta_i=\pm1)$ that is fixed point free in codimension $1$ and such that not all the $\delta_i(i=0,\cdots,k)$ are the same;
\item a Gorenstein log Del Pezzo surface of degree $4$.
\end{enumerate}
\end{thm}

In particular, every Fano variety $X$ with $\epsilon(-K_X,x)=\dim X$ lifts to characteristic zero at least when $p=\mathrm{char}(k)$ is different from 2.

As for the weak boundedness statement of Theorem \ref{thm:char 0}, we're only able to generalize it in the 3-dimensional case:

\begin{thm} \label{thm:weakbdd}
Given $\epsilon>0$ and assume that $\mathrm{char}(k)=p>5$, then the set of Fano threefolds $X$ such that $\epsilon(-K_{X},x)>2+\epsilon$ for some smooth point $x\in X$ is weakly bounded.
\end{thm}

Combining with the main result of \cite{F-seshadri}, we also have the following corollary, which partially generalizes the birational boundedness statement of Theorem \ref{thm:char 0}:

\begin{cor} \label{cor:birbdd}
Given $\epsilon>0$ and assume that $\mathrm{char}(k)=p>5$, then the set of Fano threefolds $X$ such that $\epsilon(-K_{X},x)>2+\epsilon$ for some smooth point $x\in X$ is birationally bounded.
\end{cor}

In view of \cite{me}, it is better to consider Theorem \ref{thm:weakbdd} in the context of moving Seshadri constants (see \S \ref{sec:mov_Seshadri}), which generalize the notion of Seshadri constants to arbitrary (not necessarily ample) divisors. It has the advantage of behaving better under typical operations in the Minimal Model Program (MMP). By a result of Demailly, the moving Seshadri constants and the original Seshadri constants coincide for ample divisors, hence Theorem \ref{thm:weakbdd} can be regarded as a special case of the following more general statement:

\begin{thm} \label{thm:weakbdd_m}
Given $\epsilon>0$ and assume that $\mathrm{char}(k)=p>5$, then there exists a constant $M$ depending only on $\epsilon$ such that if $X$ is a normal projective threefold with $\epsilon_{m}(-K_{X},x)>2+\epsilon$ for some smooth point $x\in X$, then $\vol(-K_X)<M$.
\end{thm}

We now outline the proof of these theorems. Let $\sigma:Y\rightarrow X$ be the blowup of $X$ at $x$ with exceptional divisor $E$ and consider $D=\sigma^*L-\epsilon(L,x)E$ where $L=-(K_X+\Delta)$ as in Theorem \ref{thm:Pn}. In characteristic zero, the proof of Theorem \ref{thm:char 0}(1) in \cite{lz16} goes by analyzing the morphism defined by $|mD|$ $(m\gg0)$. To adapt it to positive characteristic, we need to prove that $\epsilon(L,x)\in\bQ$ and that $D$ is semiample (which are somewhat obvious over $\bC$). Our observation (Corollary \ref{cor:=n->F-reg}) here is that the assumption on Seshadri constant actually implies the global $F$-regularity of the pair $(Y,\Delta)$ and this suffices to conclude the rationality of $\epsilon(L,x)$, which is essentially a consequence of Kodaira vanishing on $Y$ by the argument in \cite[Proposition 1.1]{bs}. On the other hand, the semiample-ness of $D$ is a slightly delicate application of the methods in \cite{bpf} and the key step is given by Lemma \ref{lem:basept} on the base locus of adjoint divisors. Once this is done, Theorem \ref{thm:Pn} almost follows from the same argument in \cite{lz16} while Theorem \ref{thm:=n} (resp. Theorem \ref{thm:n^n}) reduces to the classification in positive characteristic of varieties containing a projective space in the smooth locus (resp. Gorenstein conic bundles in the sense of Definition \ref{def:Gor_cb} containing the projective space as a double section) under certain conditions. These two topics are treated in \S \ref{sec:contain_Pn} and \S \ref{sec:conic_bundle} respectively and they combine to give the proof of the Theorem \ref{thm:Pn}, \ref{thm:=n} and \ref{thm:n^n} in \S \ref{sec:pf-1}.

As for the weak boundedness statement (i.e. Theorem \ref{thm:weakbdd_m}), the corresponding result in characteristic zero \cite[Theorem 3.6]{me} is proved using the Koll\'ar-Shokurov connectedness theorem, whose validity remains open in positive characteristic (even for threefolds). Our proof strategy then, is to come up with several weaker versions of the Koll\'ar-Shokurov connectedness theorem in positive characteristic and, assuming weak boundedness fails, construct a pair that violates one of these. As a preliminary step, by the work of \cite{mmp-hx,mmp-birkar,mmp-bw}, we may run the MMP for threefolds in characteristic $p>5$, and this easily reduces Theorem \ref{thm:weakbdd_m} to the case of Mori fiber spaces (Lemma \ref{lem:reduction}). Depending on the behaviour of the Mori fiber spaces, we have two cases to consider (treated in \S \ref{sec:fiber type} and \S \ref{sec:rho=1} respectively). If the Mori fiber space is of fiber type, we can prove that it is birational to a $\bP^2$-bundle over a curve and it is then relatively straightforward to construct a pair that violates connectedness theorem using an argument similar to \cite{jiang}. The other case is when we have a terminal Fano variety of Picard number one. In this case, our argument again relies heavily on the methods in \cite{bpf} and the upshot is to find some highly singular divisors in the pluri-anticanonical system that cut out a zero dimension subscheme (which can be viewed as an analog of isolated log canonical center in characteristic zero) supported at a very general point. A key tool here is a local version (see Lemma \ref{lem:global generation}) of \cite[Theorem A]{non-nef}, which allows us to construct new singular divisors out of existing ones.

This paper is organized as follows. In \S \ref{sec:prelim} we collect some definitions and useful results on $F$-singularity and moving Seshadri constants. In \S \ref{sec:at least n} we deal with Fano varieties with anticanonical Seshadri constants no smaller than the dimension and prove Theorem \ref{thm:Pn}, \ref{thm:=n} and \ref{thm:n^n}. Finally, \S \ref{sec:weakbdd} is devoted to the proof of Theorem \ref{thm:weakbdd_m}.

\subsection*{Acknowledgement}

The author would like to thank his advisor J\'anos Koll\'ar for constant support, encouragement and numerous inspiring conversations. He also wishes to thank Takumi Murayama for several useful comments on an earlier draft of this paper, Yuchen Liu for helpful discussions, and the anonymous referee for careful reading of the manuscript.

\section{Preliminary} \label{sec:prelim}

\subsection{Notation and conventions}

Unless otherwise specified, all varieties are assumed to be normal and defined over an algebraically closed field $k$ of characteristic $p\ge0$. A \emph{pair} $(X,D)$ consists of a variety $X$ and an effective $\bQ$-divisor $D$ on $X$ such that $K_X+D$ is $\bQ$-Cartier. A dominant morphism $f:X\rightarrow Y$ is called a \emph{fiber type} morphism if it has connected fibers and $0<\dim Y<\dim X$.

\subsection{$F$-singularities} \label{sec:F-sing}

We first recall a few definitions and results on singularities in characteristic $p$.

\begin{defn}
Let $X$ be a normal quasi-projective variety and $\Delta$ an effective $\bQ$-divisor on $X$. Fix a closed point $x\in X$. The pair $(X,\Delta)$ is called \emph{globally $F$-regular} (resp. \emph{globally sharply $F$-split}) if for all effective Weil divisor $D$ on $X$ (resp. for $D=0$), there exists an $e$ such that the composition
\begin{equation}\label{eq:f-reg}
\cO_X\rightarrow F_*^e\cO_X\hookrightarrow F_*^e\cO_X(\left\lceil (p^e-1)\Delta \right\rceil +D)
\end{equation}
splits as a map of $\cO_X$-modules. It is said to be \emph{strongly $F$-regular} (resp. \emph{sharply $F$-pure}) at $x$ if the pair is globally $F$-regular (resp. globally sharply $F$-split) in some affine neighbourhood of $x$.
\end{defn}

Since $X$ is quasi-projective, any effective divisor is contained in the support of some ample divisor, hence in the above definition of global $F$-regularity, it suffices to check splitting of (\ref{eq:f-reg}) when $D$ is Cartier and ample. It is also clear from the definition that if $(X,\Delta)$ is globally $F$-regular and $0\le\Delta'\le\Delta$ then $(X,\Delta')$ is also globally $F$-regular. Moreover, if $H$ is another effective divisor then $(X,\Delta+\epsilon H)$ is also globally $F$-regular for $0<\epsilon\ll 1$ and thus we can perturb the divisor $\Delta$ (preserving global $F$-regularity) so that no coefficient of $\Delta$ has a denominator divisible by $p$. For more background on global $F$-regularity, see \cite{ss}.

\begin{defn}
Let $L$ be an ample $\bQ$-divisor on $X$. Assume that $X$ is strongly $F$-regular, then the \emph{$F$-pure threshold} of the (polarized) pair $(X,L)$ is defined to be the supremum of all $t\ge0$ such that $(X,tD)$ is strongly $F$-regular for all effective $\bQ$-divisor $D\sim_\bQ L$. When $X$ is Fano, we define its $F$-pure threshold, denoted by fpt($X$), as the $F$-pure threshold of $(X,-K_X)$.
\end{defn}

By \cite[Lemma 3.4]{hw}, for any effective divisor $D$ on a normal variety $X$ and for any integer $e>0$, we have an isomorphism of $F_*^e\cO_X$-modules
\[\mathcal{H}om_{\cO_X}(F_*^e(\cO_X(D)),\cO_X)\cong F_*^e(\cO_X((1-p^e)K_X-D)).
\]
Viewing an element of $F_*^e(\cO_X((1-p^e)K_X-D))$ as a map $\theta:F_*^e(\cO_X(D))\rightarrow \cO_X$ and evaluating at $1\in F_*^e\cO_X \subseteq F_*^e(\cO_X(D))$, we obtain the trace map
\[\mathrm{Tr}_{X}^{e}(D):F_*^e(\cO_X((1-p^e)K_X-D))\rightarrow \cO_X.
\]
By abuse of notation, we will often denote $\mathrm{Tr}_{X}^{e}(D)$ simply by $\mathrm{Tr}_{X}^{e}$ or $\mathrm{Tr}^{e}$. It is quite straightforward to see that (\ref{eq:f-reg}) splits if and only if 
\[\mathrm{Tr}_{X}^{e}:F_*^e(\cO_X(\lfloor (1-p^e)(K_X+\Delta)\rfloor-D))\rightarrow \cO_X
\]
induces a surjective map on global sections (see e.g. \cite[Proposition 2.5]{bpf}).

The following criterion turns out to be quite useful when verifying a given pair is globally $F$-regular.

\begin{lem} \label{lem:F-reg}
Let $(X,D=E+\Delta)$ be a pair such that $L=-(K_{X}+D)$ is nef and big, $E$ is a prime divisor contained in the smooth locus of $X$ and $E\not\in\mathrm{Supp}(\Delta)$. Assume that $(E,\Delta|_{E})$ is globally $F$-regular and $L|_{E}$ is ample, then $(X,\Delta)$ is also globally $F$-regular.
\end{lem}

\begin{proof}
We first make a few reductions. Since $L$ is nef and big, there exists an effective divisor $M$ such that $L-\epsilon M$ is ample for all $0<\epsilon\ll1$. As $L|_{E}$ is ample, $(L+\epsilon E)|_{E}$ is also ample for sufficiently small $\epsilon$, hence $L+\epsilon E$ is nef and big for $0\le\epsilon\ll1$ (if $C$ is a curve such that $(L+\epsilon E\cdot C)<0$ then since $L$ is nef we have $C\subseteq E$, but this contradicts the ampleness of $(L+\epsilon E)|_{E}$). Let $a\ge0$ be the coefficient of $E$ in $M$, let $\lambda=\frac{1}{a+1}$ and $D'=D+\epsilon(\lambda M-(1-\lambda)E)$, then $E$ still has coefficient one in $D'$ (i.e. $D'=E+\Delta'$ where $E\not\in\mathrm{Supp}(\Delta')$) and for sufficiently small $\epsilon$, $(E,\Delta'|_{E})$ is still globally $F$-regular. We also have $-(K_{X}+D')=(1-\lambda)(L+\epsilon E)+\lambda(L-\epsilon M)$, hence for $0<\epsilon\ll1$, $-(K_{X}+D')$ is ample. Since $\Delta'\ge\Delta$, we may replace $D$ by $D'$ and assume that $L=-(K_{X}+D)$ is ample in what follows. By perturbing the coefficients of components of $\Delta$, we may also assume that $(p^{e}-1)\Delta$ has integral coefficients for some $e>0$.

Let $H$ be an ample Cartier divisor on $X$ such that $\Delta\cup\mathrm{Sing}(X)\subseteq\mathrm{Supp}(H)\not\subseteq E$. Consider the following commutative diagram 
\[
\xymatrix{
 F_{*}^{e}\mathcal{O}_{X}((1-p^{e})(K_{X}+E+\Delta)-H)\ar[r]^-{i}\ar[d]^{\mathrm{Tr}_{X}^{e}} & F_{*}^{e}\mathcal{O}_{E}((1-p^{e})(K_{E}+\Delta_{E})-H)\ar[d]^{\mathrm{Tr}_{E}^{e}}\\
 \mathcal{O}_{X}\ar[r] & \mathcal{O}_{E}
}
\]
where the two vertical arrows are given by the trace map. By assumption, $\mathrm{Tr}_{E}^{e}$ induces a surjection on global sections. As $L$ is ample, for sufficiently large and divisible $e$ we have $H^{1}(X,F_{*}^{e}\mathcal{O}_{X}((1-p^{e})(K_{X}+E+\Delta)-E-H))=H^{1}(X,(p^{e}-1)L-E-H)=0$, thus by the long exact sequence of cohomology, $i$ also induces a surjection on global sections. It follows that $H^{0}(\mathrm{Tr}_{X}^{e})$ is surjective as well. By \cite[Theorem 3.9]{ss}, this implies that $(X,\Delta)$ is globally $F$-regular.
\end{proof}

\begin{cor} \label{cor:=n->F-reg}
Let $(X,\Delta)$ be a pair such that $L=-(K_{X}+\Delta)$ is ample.
Assume that $(L^{n})>n^{n}$ and $\epsilon(L,x)\ge n$ for some smooth
point $x\in X\backslash\Delta$. Let $Y$ be the blow up of $X$ at
$x$ and let $\Delta_Y$ be the strict transform of $\Delta$ on $Y$. Then $(Y,\Delta_Y)$
is globally $F$-regular.
\end{cor}

\begin{proof}
Let $E$ be the exceptional divisor of the blowup $\sigma:Y\rightarrow X$,
then the pair $(Y,E+\Delta_Y)$ satisfies all the assumptions of the
Lemma \ref{lem:F-reg}.
\end{proof}

\begin{rem}
Note that $\epsilon(L,x)\ge n$ already implies $(L^{n})\ge n^{n}$.
However, the assumption $(L^{n})>n^{n}$ in the above corollary can
not be removed in general (even in the boundary free case, i.e. when $\Delta=0$). For example, consider the pair $(X=\bP^n,H)$ where $H$ is a hyperplane, then clearly $\epsilon(-(K_X+H),x)=n$ for any smooth point $x\in X$, but $H$ is the center of $F$-purity of the pair. As another example, consider the Fermat cubic surface $Y=(x^{3}+y^{3}+z^{3}+w^{3}=0)\subseteq\mathbb{P}^{3}$, then $Y$ is not even globally $F$-split in characteristic 2, but $Y$ is also the blow up of a smooth del Pezzo surface of degree 4 whose anticanonical divisor has Seshadri constant 2 at the point we blow up.
\end{rem}

One of the advantages of global $F$-regularity is that most vanishing results that hold in characteristic zero remain valid. In particular we have (see \cite[Theorem 6.8]{ss} for the dual statement):

\begin{lem} \label{lem:vanishing}
Let $(Y,\Delta)$ be a globally $F$-regular pair and $D$ an effective Weil divisor on $Y$ such that $D-(K_{Y}+\Delta)$ is nef and big. Then $H^{i}(Y,\cO_Y(D))=0$ for all $i>0$.
\end{lem}

\begin{proof}
We note that the assumption implicitly requires that $D-(K_{Y}+\Delta)$ is $\bQ$-Cartier. We may perturb the pair as before and assume that $(p^{e}-1)\Delta$ has integral coefficients for sufficiently divisible $e$ and that $D-(K_{Y}+\Delta)$ is ample. Let $q=p^{e}$. Since $(Y,\Delta)$ is globally $F$-regular, the trace map 
\[
\mathrm{Tr}^{e}:F_{*}^{e}\mathcal{O}_{Y}((1-q)(K_{Y}+\Delta))\hookrightarrow F_{*}^{e}\mathcal{O}_{Y}((1-q)K_{Y})\rightarrow\mathcal{O}_{Y}
\]
splits for every sufficiently divisible $e$. Taking the reflexive tensor with $\cO_Y(D)$
we see that $H^{i}(Y,\cO_Y(D))$ is a direct summand of $H^{i}(Y,\cO_Y(D)\otimes F_{*}^{e}\mathcal{O}_{Y}((1-q)(K_{Y}+\Delta)))=H^{i}(Y,\mathcal{O}_{Y}((1-q)(K_{Y}+\Delta)+qD))$,
but the latter group is zero when $i>0$ and $q\gg0$ by Serre vanishing, thus $H^{i}(Y,\cO_Y(D))=0$.
\end{proof}

\begin{cor} \label{cor:relvanish}
Let $(Y,\Delta)$ be a globally $F$-regular pair, $f:Y\rightarrow X$ a proper morphism and $D$ an Weil divisor on $Y$ such that $D-(K_{Y}+\Delta)$ is $f$-nef and $f$-big. Then $R^i f_*\cO(D)=0$ for all $i>0$.
\end{cor}

\begin{proof}
It suffices to show that $H^i(Y,\cO_Y(D+f^*H))=0$ for sufficiently ample divisor $H$ on $X$. But for such $H$, $D+f^*H-(K_Y+\Delta)$ is nef and big by assumption, so the statement follows directly from Lemma \ref{lem:vanishing}.
\end{proof}

\subsection{Test ideals} \label{sec:test ideals}

Next we review the definition of test ideals on a smooth variety (this will be the only case we need and we refer to \cite{test_ideal} for a general treatment of test ideals). 

Our definition is taken from \cite{bms}. Let $X$ be a smooth variety and $\mathfrak{a}$ an ideal sheaf on $X$. Let $e$ be a positive integer. Let $\mathfrak{a}^{[1/p^e]}$ denote the unique smallest ideal sheaf $\cJ$ such that $\mathfrak{a}\subseteq\cJ^{[p^e]}$. One can show that (see \cite[Proposition 2.5 and Lemma 2.8]{bms})
\[\mathrm{Tr}^e_X(F^e_*(\mathfrak{a}\cdot \omega_X))=\mathfrak{a}^{[1/p^e]}\cdot \omega_X
\]
and that if $c\ge 0$ then
\[(\mathfrak{a}^{\lceil cp^e \rceil})^{[1/p^e]}\subseteq (\mathfrak{a}^{\lceil cp^{e+1} \rceil})^{[1/p^{e+1}]}.
\]

\begin{defn}
Given $\mathfrak{a}$ and $c\ge0$ as above. The \emph{test ideal} of the pair $(X,\mathfrak{a}^c)$ is defined to be
\[\tau(X,\mathfrak{a}^c)=\bigcup_{e\ge0} (\mathfrak{a}^{\lceil cp^e \rceil})^{[1/p^e]}.
\]
\end{defn}

Since $\cO_X$ is Noetherian, this is well-defined and by the previous discussion we have \[\tau(X,\mathfrak{a}^c)\cdot \omega_X=\mathrm{Tr}^e_X(F^e_*(\mathfrak{a}^{\lceil cp^e \rceil}\cdot \omega_X))
\]
for $e\gg0$. If $\mathfrak{a}=\cO_X(-D)$ for some divisor $D\ge0$, we simply write $\tau(X,\mathfrak{a}^c)$ as $\tau(X,\Delta)$ where $\Delta=cD$.

The following property of test ideals is well known to expert. It is analogous to the corresponding property of multiplier ideals in characteristic zero.

\begin{lem} \label{lem:mult and test ideal}
Let $X$ be a smooth variety of dimension $n$. Let $x\in X$ and let $D$ be an effective $\bQ$-divisor on $X$.
\begin{enumerate}
    \item If $\mult_x D<1$, then the pair $(X,D)$ is strongly $F$-regular at $x$. In particular, $\tau(X,D)_x=\cO_{X,x}$.
    \item If $\mult_x D=N\ge n$, then $\tau(X,D)\subseteq\mathfrak{m}_x^{\lfloor N-n+1 \rfloor}$. In particular, $(X,D)$ is not strongly $F$-regular at $x$.
\end{enumerate}
\end{lem}

\begin{proof}
(2) is a consequence of \cite[Proposition 3.3]{non-nef}, so we only need to prove (1). If $\cO_X(-\lceil p^e D \rceil)\subseteq \mathfrak{m}_x^{[p^e]}$ for $e\gg0$ then $\mult_x D\ge 1$, hence if $\mult_x D<1$ then $\tau(X,D)_x=\cO_{X,x}$ by definition. If follows that $(X,(1+\epsilon)D)$ is sharply $F$-pure for some $0<\epsilon\ll 1$. By \cite[Theorem 3.9]{ss}, this implies that $(X,D)$ is strongly $F$-regular at $x$.
\end{proof}

\begin{cor} \label{cor:fpt-Pn}
Let $X=\bP^n$ and let $D$ be an effective $\bQ$-divisor of degree $<1$ on $X$. Then $(X,D)$ is globally $F$-regular.
\end{cor}

\begin{proof}
This follows from Lemma \ref{lem:mult and test ideal}(1) by taking the cone over $(X,D)$.
\end{proof}

\subsection{Moving Seshadri constants} \label{sec:mov_Seshadri}

We now recall some results on moving Seshadri constants and carry out the reduction step of Theorem \ref{thm:weakbdd_m}. Essentially all the results in this section are contained in \cite{me} as the same proofs there work in any characteristic.

\begin{defn} \label{defn:moving Seshadri} \cite{nakamaye,elmn}
Let $D$ be a $\bQ$-divisor on a normal variety $X$. Let $x\in X$ be a smooth point. The moving Seshadri constant of $D$ at $x$ is defined to be 
\[\epsilon_m(D,x)=\limsup_k\frac{s(kD,x)}{k}
\]
where the limit is taken over sufficiently large and divisible integers $k$ and $s(kD,x)$ is the largest integer $s\ge -1$ such that the natural map
\[H^{0}(X,\cO_X(kD))\rightarrow H^{0}(X,\cO_X(kD)\otimes\cO_X/\mathfrak{m}_{x}^{s+1})
\]
is surjective. We also define
\[\epsilon_m(D)=\sup_{x\in X^{\circ}}\epsilon_m(D,x)
\]
where $X^\circ$ is the smooth locus of $X$. It is not hard to see that the supremum is actually a maximum and is achieved at a very general point of $X$.
\end{defn}

Here are some of the basic properties of moving Seshadri constants.

\begin{lem} \label{lem:restrict_m} 
Let $L$ be a $\bQ$-Cartier $\bQ$-divisor on $X$, let $\pi:Y\to X$ be a morphism, and let $y$ be a smooth point of $Y$ such that $X$ is smooth at $\pi(y)$ and $\pi$ is an immersion at $y$. Then $\epsilon_m(L,\pi(y))\le\epsilon_m(\pi^*L,y)$.
\end{lem}

\begin{proof}
This follows from the same proof of \cite[Lemma 3.1]{me}.
\end{proof}

\begin{lem} \label{lem:nondecrease}
Let $\phi:X\dashrightarrow Y$ be a birational contraction between normal varieties, then $\epsilon_{m}(-K_{X})\le\epsilon_{m}(-K_{Y})$ and $\vol(-K_X)\le\vol(-K_Y)$.
\end{lem}

\begin{proof}
The first statement is simply \cite[Corollary 3.3]{me} while the second follows from the injection $H^0(X,-mK_X)\rightarrow H^0(Y,-mK_Y)$ induced by $\phi$.
\end{proof}

\begin{defn} \cite[Theorem 1.33]{mmp}
A projective birational morphism $\phi:Y\rightarrow X$ is called a \emph{terminal modification} of $X$ if $Y$ is $\bQ$-factorial, terminal and $K_Y$ is $\phi$-nef.
\end{defn}

The existence of terminal modification is a formal consequence of the MMP, so by the work of \cite{mmp-hx,mmp-birkar,mmp-bw}, terminal modification exists for threeholds in characteristic $p>5$.

\begin{lem} \label{lem:mmodel}
Let $\phi:Y\rightarrow X$ be a terminal modification of $X$, then $\epsilon_m(-K_X)=\epsilon_m(-K_Y)$ and $\vol(-K_X)=\vol(-K_Y)$.
\end{lem}

\begin{proof}
The first equality follows from \cite[Lemma 3.4]{me}. To see the second equality, let $D_X\in|-mK_X|$ and $D_Y$ its strict transform on $Y$. We may write $mK_Y+D_Y+E_Y\sim\phi^*(mK_X+D_X)\sim 0$ for some $\phi$-exceptional divisor $E_Y$. Note that $E_Y$ has integral coefficients. Apply \cite[Lemma 2.5]{me} to the pair $(X,\frac{1}{m}D)$ we see that $E_Y$ is effective. It follows that $D_Y+E_Y\in|-mK_Y|$ and we have an injection $\phi^{-1}_*:H^0(X,-mK_X)\rightarrow H^0(Y,-mK_Y)$. On the other hand $\phi_*$ also induces an inclusion $H^0(Y,-mK_Y)\rightarrow H^0(X,-mK_X)$ and $\phi_*\circ\phi^{-1}_*=\mathrm{id}$, hence it's an isomorphism and $\vol(-K_X)=\vol(-K_Y)$.
\end{proof}

Using these properties we can easily reduce the proof of Theorem \ref{thm:weakbdd_m} to the case of Mori fiber spaces.

\begin{defn}
Let $X$ be a normal variety and $f:X\rightarrow Y$ a projective morphism with $f_*\cO_X=\cO_Y$. Then $f$ is called a \emph{Mori fiber space} if
\begin{enumerate}
    \item $X$ has $\bQ$-factorial terminal singularities,
    \item the relative Picard number $\rho(X/Y)=1$, and
    \item $-K_X$ is $f$-ample.
\end{enumerate}
\end{defn}

\begin{lem} \label{lem:reduction}
It suffices to prove Theorem \ref{thm:weakbdd_m} when the threefold $X$ admits a Mori fiber space structure.
\end{lem}

\begin{proof}
Let $Y\rightarrow X$ be a terminal modification of $X$ and run the $K_Y$-MMP on $Y$. Since $K_X$ is not pseudoeffective by assumption, the MMP ends with $Y\dashrightarrow Y_1$ where $Y_1$ admits a Mori fiber space structure. By Lemma \ref{lem:nondecrease} and \ref{lem:mmodel}, we have $\epsilon_m(-K_{Y_1})\ge\epsilon_m(-K_X)$ and $\vol(-K_{Y_1})\ge\vol(-K_X)$. Thus if Theorem \ref{thm:weakbdd_m} holds for Mori fiber spaces then it holds for all threefolds as well.
\end{proof}

\section{Fano varieties with Seshadri constants at least $n$} \label{sec:at least n}

In this section, we study Fano varieties with anticanonical Seshadri constants no smaller than their dimension and in particular prove Theorem \ref{thm:Pn}, \ref{thm:=n} and \ref{thm:n^n}. Since this eventually reduces to the classification of certain varieties containing a projective space in the smooth locus or  Gorenstein conic bundles containing the projective space as a double section, we first give a treatment of these two topics.

\subsection{Varieties containing projective space as a divisor} \label{sec:contain_Pn}

In \cite{lz16}, an important step in the classification of varieties $X$ with $\epsilon(-K_X,x)\ge n$ is the classification of varieties (over $\bC$) that contain a divisor $D\cong\bP^{n-1}$ in the smooth locus. In this section we carry out the parallel study of such varieties in positive characteristic. We start with the Picard number one case.

\begin{lem} \label{lem:picsurj}
Let $X$ be a normal projective variety of dimension $n$ and $D\cong\mathbb{P}^{n-1}$ a divisor contained in its smooth locus. Assume that $\mathcal{N}_{D/X}$ is nef and $n\ge3$ if $\mathcal{N}_{D/X}$ is ample. Then the natural restriction $\mathrm{Cl}(X)\rightarrow\mathrm{Cl}(D)$
is surjective.
\end{lem}

\begin{proof}
Let $d=\deg\mathcal{N}_{D/X}$. If $d>0$, let $Z\subseteq D$ be a smooth hypersurface of degree
$d$ and let $\tilde{X}$ be the blow up of $X$ along $Z$. Note
that since $n\ge3$, $Z$ is connected. Let $E$ be the exceptional
divisor and $\tilde{D}$ the strict transform of $D$. Then we have
$\mathrm{Cl}(\tilde{X})\cong\mathrm{Cl}(X)\oplus\mathbb{Z}[E]$ and
the image of $\mathrm{Cl}(\tilde{X})\rightarrow\mathrm{Cl}(\tilde{D})$
is the same as the image of $\mathrm{Cl}(X)\rightarrow\mathrm{Cl}(D)\cong\mathrm{Cl}(\tilde{D})$.
Since $\mathcal{N}_{\tilde{D}/\tilde{X}}\cong\mathcal{O}_{\tilde{D}}$,
we may replace $(X,D)$ by $(\tilde{X},\tilde{D})$ and reduce to
the case that $d=0$.

As $D\cong\mathbb{P}^{n-1}$ and $d=0$, we have $h^{0}(D,\mathcal{N}_{D/X})=1$
and $h^{1}(D,\mathcal{N}_{D/X})=0$, hence the Hilbert scheme of $X$
is smooth and of dimension 1 at the point $[D]$. It follows
that there exists a curve $C$ (not necessarily proper) and a family
of divisors of $X$
\[
\xymatrix{Y\ar[d]_{f}\ar[r]^{g} & X\\
C
}
\]
such that $f$ is smooth, $g$ identifies a fiber $F$ of $f$ with
$D$ and if $F_{s},F_{t}$ are fibers of $f$ over $s\neq t\in C$,
then $g(F_{s})\neq g(F_{t})$. As $\mathbb{P}^{n-1}$ is rigid, after shrinking $C$ we may assume that
all fibers of $f$ are isomorphic to $\mathbb{P}^{n-1}$; moreover since $C$ is a curve, $f$ is indeed a $\mathbb{P}^{n-1}$-bundle by Tsen's theorem.
On the other hand, as $\mathcal{N}_{D/X}\cong\mathcal{O}_{D}$, we
have $D'\cap D=\emptyset$ if $D'\neq D$ is algebraically equivalent
to $D$, thus $g$ is an isomorphism in a neighbourhood of $D$. Therefore,
after further shrinking of $C$ we may assume that $g$ is an open immersion.
Then $g^*:\mathrm{Cl}(X)\rightarrow\mathrm{Cl}(Y)$ is surjective. Since
$f:Y\rightarrow C$ is a $\mathbb{P}^{n-1}$-bundle, $\mathrm{Cl}(Y)\rightarrow\mathrm{Cl}(F)$
is also surjective, so the lemma follows.
\end{proof}

\begin{rem}
The $n\ge3$ assumption in the above lemma is necessary if $\mathcal{N}_{D/X}$ is ample, since $\mathrm{Cl}(X)\rightarrow\mathrm{Cl}(D)$ is not surjective when $D$ is a conic in $X=\mathbb{P}^{2}$. It is also not hard to see that the statement does not hold if $\mathcal{N}_{D/X}$ has negative degree. For example, consider a general surface $S$ of degree $d\ge 4$ that contains a conic curve $C$, then $\Pic(S)$ is generated by $C$ by \cite[Theorem II.3.1]{lopez} and the hyperplane class $H$. Since $(C\cdot H)=2$ and $(C^2)=2(3-d)$, $\Pic(S)\rightarrow\Pic(C)$ is not surjective and the image is an index 2 subgroup.
\end{rem}

\begin{lem} \label{lem:rho1}
Let $X$ be a normal projective variety of dimension $n\ge2$ containing
a divisor $D\cong\mathbb{P}^{n-1}$ in its smooth locus. Assume that
$\rho(X)=1$, then one of the following holds:
\begin{enumerate}
\item $X\cong\mathbb{P}(1^{n},d)$ for some $d\in\mathbb{Z}_{>0}$ and $D$
is the hyperplane defined by the vanishing of the last coordinate;
or
\item $n=2$, $X\cong\mathbb{P}^{2}$ and $D$ is a smooth conic.
\end{enumerate}
\end{lem}

\begin{proof}
First consider the case $n\ge3$. Let $X_{0}$ be the smooth
locus of $X$. Since $\rho(X)=1$, $D$ is ample, so $X$ has only
isolated singularities and the natural map $\mathrm{Cl}(X)\cong\mathrm{Pic}(X_{0})\rightarrow\mathrm{Pic}(\hat{X})$
is an isomorphism by \cite[Expos\'e XI, Proposition 2.1]{sga2}, where $\hat{X}$ is the
formal conpletion of $X$ along $D$. As $D\cong\mathbb{P}^{n-1}$
and $n\ge3$, we have $H^{1}(D,\mathcal{O}_{D}(-mD))=0$ for all $m$,
hence by the exact sequence (c.f. \cite[Expos\'e XI, \S 1]{sga2})
\begin{equation}
H^{1}(D,\mathcal{O}_{D}(-mD))\rightarrow\mathrm{Pic}(D_{m+1})\rightarrow\mathrm{Pic}(D_{m})\rightarrow H^{2}(D,\mathcal{O}_{D}(-mD))
\end{equation}
the restriction map $\mathrm{Pic}(\hat{X})\rightarrow\mathrm{Pic}(D)$
is injective; on the other hand it is also surjective by Lemma \ref{lem:picsurj},
thus we have an isomorphism $\mathrm{Cl}(X)\cong\mathrm{Pic}(D)$.
In particular, $X$ is $\mathbb{Q}$-factorial and since $-(K_{X}+D)|_{D}=-K_{D}$
is ample, $-(K_{X}+D)$ is ample on $X$ itself. By Lemma \ref{lem:F-reg},
$X$ is globally $F$-regular. 

Let $H$ be the ample generator of $\mathrm{Cl}(X)$, then $\mathcal{O}_{D}(H)\cong\mathcal{O}_{D}(1)$
and there exists a positive integer $d$ such that $D\sim dH$. Consider
the exact sequence
\begin{equation}
H^{0}(X,\mathcal{O}_{X}(H-D))\rightarrow H^{0}(X,\mathcal{O}_{X}(H))\rightarrow H^{0}(D,\mathcal{O}_{D}(H))\rightarrow H^{1}(X,\mathcal{O}_{X}(H-D))\label{eq:H0(H)}
\end{equation}
Since $X$ is globally $F$-regular and $H-(K_{X}+D)$ is ample, we have
$H^{1}(X,\mathcal{O}_{X}(H-D))=0$ by Lemma \ref{lem:vanishing}.
If $d=1$, then $H\sim D$ is Cartier and it follows from (\ref{eq:H0(H)})
that $H$ is globally generated and $h^{0}(X,H)=n+1$, thus $|H|$
induces a morphism $X\rightarrow\mathbb{P}^{n}$ of degree $(H^{n})=(H^{n-1}\cdot D)=1$,
which is an isomorphism $X\cong\mathbb{P}^{n}$. If $d>1$, then $H^{0}(X,\mathcal{O}_{X}(H-D))=0$
and by (\ref{eq:H0(H)}) we have $h^{0}(X,H)=n$ and is globally generated
in a neighbourhood of $D$. The global sections of $\mathcal{O}_{X}(H)$
and the canonical section of $\mathcal{O}_{X}(D)\cong\mathcal{O}_{X}(dH)$
then defines a morphism $X\rightarrow\mathbb{P}(1^{n},d)$ of degree
$(H^{n-1}\cdot D)=1$, which is again an isomorphism $X\cong\mathbb{P}(1^{n},d)$.
By construction, $D$ is identified with the hyperplane defined by
the vanishing of the last coordinate.

Next assume $n=2$. By assumption $(D^{2})>0$, thus by Lemma \ref{lem:Q-factorial},
$X$ is $\mathbb{Q}$-factorial. As in the $n\ge3$ case, we still
have $-(K_{X}+D)$ is ample and $X$ is globally $F$-regular. If $\mathrm{Cl}(X)\rightarrow\mathrm{Pic}(D)$
is surjective then as before we have $X\cong\mathbb{P}(1,1,d)$ for
some $d>0$. If $\mathrm{Cl}(X)\rightarrow\mathrm{Pic}(D)$ is not
surjective then since $(K_{X}+D\cdot D)=-2$ we see that the image
is generated by the restriction of $H=-(K_{X}+D)$. We also have $D\sim_{\mathbb{Q}}dH$
for some $d>0$. We now divide into three cases according to the value
of $d$. 

If $d=1$, then by the same argument as in the $n\ge3$ case, $|D|$
is base point free and identifies $X$ with a quadric in $\mathbb{P}^{3}$
(note that $(D^{2})=2d=2$ and $h^{0}(X,D)=4$) and $D$ a hyperplane
section. Since $\rho(X)=1$, $X$ is singular, but then $\mathrm{Cl}(X)\rightarrow\mathrm{Pic}(D)$
is surjective, contrary to our assumption.

If $d=2$, then $(H^{2})=\frac{1}{d}(H\cdot D)=1$. As before, by
(\ref{eq:H0(H)}) and the global $F$-regularity of $X$ we have $H^{0}(X,\mathcal{O}_{X}(H))\cong H^{0}(D,\mathcal{O}_{D}(2))$
and $h^{0}(X,H)=3$, hence for any $x\in D$, we may choose two different
$H_{1},H_{2}\sim H$ passing through $x$. Clearly both $H_{i}$ are
integral (otherwise $\mathrm{Cl}(X)\rightarrow\mathrm{Pic}(D)$ is
surjective). Since $(H_{1}\cdot H_{2})=(H^{2})=1$, we see that $H_{1}$
only intersects $H_{2}$ at $x$. It follows that $H$ is Cartier,
base point free and defines a morphism $X\rightarrow\mathbb{P}^{2}$
of degree $1$, which is an isomorphism that identifies $D$ with
a smooth conic.

Finally if $d\ge3$, we still have $H^{0}(X,\mathcal{O}_{X}(H))\cong H^{0}(D,\mathcal{O}_{D}(2))$.
Let $s_{0}$ be the canonical section of $H^{0}(X,\mathcal{O}_{X}(D))$.
Choose $s_{1},s_{2}\in H^{0}(X,\mathcal{O}_{X}(H))$ whose restrictions
on $D$ induce a separable morphism $D\rightarrow\mathbb{P}^{1}$
of degree 2. Then we can define a separable double cover $f:X\rightarrow Y=\mathbb{P}(1,1,d)$
sending $x\in X$ to $[s_{1}(x):s_{2}(x):s_{0}(x)]$. We have $K_{X}=f^{*}K_{Y}+R$
for some divisor $R$ supported in the branched locus of $f$. A direct
calculation yields $R\sim_{\mathbb{Q}}H$, thus $f_{*}R\sim f_{*}H\sim2L$
where $L$ is the ample generator of $\mathrm{Cl}(Y)\cong\mathbb{Z}$.
But since $d\ge3$, $f_{*}R$ and thus $R$ cannot be integral. It
follows that we have a decomposition $H\sim_{\mathbb{Q}}R_{1}+R_{2}$
for some effective nonzero $\mathbb{Z}$-divisor $R_{1},R_{2}$, but
then $(R_{1}\cdot D)+(R_{2}\cdot D)=(H\cdot D)=2$ and $\mathrm{Cl}(X)\rightarrow\mathrm{Pic}(D)$
is surjective. So this case cannot happen and the proof is now complete.
\end{proof}

The following lemma is used in the above proof.

\begin{lem} \label{lem:Q-factorial}
Let $X$ be a normal projective surface. Suppose
there exists a smooth rational curve $C$ contained in the smooth
locus of $X$ such that $(C^{2})\ge0$. Then $X$ has rational singularities.
In particular, $X$ is $\mathbb{Q}$-factorial.
\end{lem}

\begin{proof}
After possibly blowing up points on $C$ we reduce to the case that
$(C^{2})=0$. Let $\tilde{X}\rightarrow X$ be the minimal resolution
of $X$ and let $\tilde{C}$ also denote its strict transform on $\tilde{X}$.
Since $C$ is a smooth rational curve we have $(K_{\tilde{X}}\cdot\tilde{C})=-2$
by adjunction. By Riemann-Roch we have 
\[
\chi(\mathcal{O}_{\tilde{X}}(m\tilde{C}))=\frac{1}{2}(m\tilde{C}\cdot m\tilde{C}-K_{\tilde{X}})+\chi(\mathcal{O}_{\tilde{X}})=m+\chi(\mathcal{O}_{\tilde{X}})
\]
On the other hand by Serre duality we have $h^{2}(\tilde{X},m\tilde{C})=h^{0}(\tilde{X},K_{\tilde{X}}-m\tilde{C})=0$
when $m\gg0$. It follows that $h^{0}(\tilde{X},m\tilde{C})\ge2$
for sufficiently large $m$. Hence there exists an effective divisor
$\Gamma\sim m\tilde{C}$ for some $m>0$ such that $\tilde{C}\not\subseteq\mathrm{Supp}(\Gamma)$.
As $(\tilde{C}\cdot\Gamma)=m(C^{2})=0$, we see that $\Gamma$ is
disjoint from $\tilde{C}$, thus $m\tilde{C}$ is base point free.
Since $\tilde{C}$ is the pullback of $C$, $C$ is semiample and
induces a morphism $p:X\rightarrow Y$ with connected fibers to a
curve $Y$ such that the general fiber is isomorphic to $C$ (if $\Gamma\equiv mC$
is an irreducible fiber in the smooth locus of $X$, then $2p_{a}(\Gamma)-2=(K_{X}+\Gamma\cdot\Gamma)=-2m$,
thus $p_{a}(\Gamma)=0$ and $m=1$). By \cite[Theorem 2 and Remark 3]{cheltsov},
$X$ has rational singularities and hence is $\mathbb{Q}$-factorial
by \cite[Proposition 17.1]{lipman}.
\end{proof}

We next turn to the case when the Picard number is at least two. In \cite[Lemma 12]{lz16}, this is done by running MMP, which is not yet available in positive characteristic in general. Nevertheless, the following lemma, which is later used to prove the base-point-freeness of certain line bundles, serves as a substitute at least for the purpose of this section.

\begin{lem} \label{lem:basept}
Let $(Y,\Delta)$ be a strongly $F$-regular pair and $D$ a nef divisor such that $D-(K_Y+\Delta)$ is nef and big. Suppose that $y\in Y$ is contained in the stable base locus of $D$, then there exists a positive dimensional subvariety $V\subseteq Y$ containing $y$ such that $D|_V$ is numerically trivial.
\end{lem}

\begin{proof}
This is indeed a consequence of the arguments in \cite[\S 3-4]{bpf}. Namely, if $y\in Y$ is not contained in any positive dimensional subvariety $V\subseteq Y$ such that $D|_V$ is numerically trivial, then the same argument as in \cite[Theorem 3.7]{bpf} creates a $\bQ$-divisor $D^{(e)}=\sum_{i=1}^n t_i(e)D_i$ and an isolated non-$F$-pure center $W$ supported at $y$ for which the proof of \cite[Theorem 1.1]{bpf} can be used to show that $y$ is not a base point of $|mD|$ for $m\gg0$.
\end{proof}

We now briefly explain the idea for classifying varieties $X$ containing a divisor $D\cong\bP^{n-1}$ such that $\rho(X)\ge2$ and $-(K_X+D)$ is ample. Instead of running the MMP, we consider divisors of the form $L_\lambda=-(K_X+\lambda D)$ and hope that for some $\lambda$, the corresponding divisor $L_\lambda$ defines the contraction of the extremal ray we want. A natural idea is to take the largest $\lambda$ such that $L_\lambda$ is nef. To make the argument work, we need to show that $\lambda\in\bQ$ and that $L_\lambda$ is semiample. Once this is done, it is quite straightforward to finish the classification.

For the next couple lemmas, we introduce the following
notations. Let $D$ be a prime divisor on $X$, we define 
\[
\rho(X,D):=\mathrm{rank}\,\mathrm{Im}(\mathrm{Pic}(X)\rightarrow\mathrm{Pic}(D))
\]
If in addition $D$ is Cartier and $L$ is an ample line bundle on $X$,
we define
\[
\epsilon(L,D)=\sup\{t\,|\,L-tD\;\mathrm{is}\;\mathrm{ample}\}
\]
and let $s(L,D)$ be the largest integer $s\ge 0$ such that $(L-sD)|_{D}$
is base point free and $H^{0}(X,L-sD)\rightarrow H^{0}(D,L-sD|_{D})$
is surjective. By convention, we set $s(L,D)=-1$ if such integer $s$ doesn't exists.

\begin{lem} \label{lem:s(l,d)}
Let $L$ be an ample line bundle on $X$ and $D$ a prime Cartier divisor. Then for all $m\ge1$ we have
\[
\frac{s(mL,D)}{m}\le\epsilon(L,D)=\lim_{m\rightarrow\infty}\frac{s(mL,D)}{m}
\]
\end{lem}

\begin{proof}
The proof is similar to that of the analogous statement for Seshadri
constants (where $D$ is the exceptional divisor of a blow up). We
first prove the inequality $\frac{s(mL,D)}{m}\le\epsilon(L,D)$. Let
$s=s(mL,D)$, it suffices to show that $mL-sD$ is nef. Suppose it
is not, then there exists a curve $C\subseteq X$ such that $(mL-sD\cdot C)<0$.
Since $L$ is ample, we have $(D\cdot C)>0$ and therefore, $C$ intersects
$D$. Choose $x\in C\cap D$. By the definition of $s(L,D)$, there
exists a section $u\in H^{0}(X,mL-sD)$ that does not vanish at $x$.
But this implies $(mL-sD\cdot C)\ge0$, a contradiction.

Now let $\lambda$ be any rational number such that $\lambda<\epsilon(L,D)$.
We will show $s(mL,D)\ge\left\lfloor \lambda m\right\rfloor $ for
$m\gg0$, thus proving the equality part of the lemma. To this end
fix $m\gg0$ and let $s=\left\lfloor \lambda m\right\rfloor $. By
Lemma \ref{lem:fujita}, $mL-sD$ is very ample and $H^{1}(X,mL-(s+1)D)=0$.
Therefore, $(mL-sD)|_{D}$ is base point free and by the long exact
sequence of cohomology, $H^{0}(X,mL-sD)\rightarrow H^{0}(D,(mL-sD)|_{D})$
is surjective. Thus $s(mL,D)\ge\left\lfloor \lambda m\right\rfloor $
and we are done.
\end{proof}

Recall the following Fujita-type result that is used in the above proof (it will also be used later).

\begin{lem} \label{lem:fujita}
Let $L$ be an ample line bundle on $X$ and let $D$ be a
Cartier divisor. Let $\lambda>0$ be such that $L-\lambda D$ is still
ample and let $m,s\ge0$ be integers such that $s\le\lambda m$. Let
$\mathcal{F}$ be a coherent sheaf on $X$. Then for $m\gg0$, $mL-sD$
is very ample and $H^{1}(X,\mathcal{F}(mL-sD))=0$. 
\end{lem}

\begin{proof}
We may assume $\lambda\in\mathbb{Q}$ (otherwise enlarge $\lambda$
slightly). Choose sufficiently large and divisible $N$ such tha $H_{1}=NL$
and $H_{2}=N(L-\lambda D)$ are both very ample. Then since $\lambda$
is rational, there exists finitely many line bundles $L_{i}$ such
that $mL-sD=L_{i}+a_{1}H_{1}+a_{2}H_{2}$ for some $i$ and some integers
$a_{1},a_{2}\ge0$. As $m\gg0$ we have $\max\{a_{1},a_{2}\}\gg0$,
thus the lemma follows from \cite[Theorem 2 and Corollary 3]{fujita}.
\end{proof}

\begin{lem} \label{lem:lambda}
Let $(X,\Delta)$ be a globally $F$-regular pair
and $D$ a prime Cartier divisor on $X$ such that $L=-(K_{X}+\Delta+D)$
is ample. Let $\lambda=\epsilon(L,D)$. Assume either $\rho(X,D)=1$
or $(L-\lambda D)|_{D}$ is ample. Then $\lambda\in\mathbb{Q}$ and
$L-\lambda D$ is semiample.
\end{lem}

\begin{proof}
We first prove $\lambda\in\mathbb{Q}$. Suppose this is not the case.
If $\rho(X,D)=1$, then as $\lambda\not\in\mathbb{Q}$, $(L-\lambda D)|_{D}$
is nef but not numerically trivial, so we reduce to the case when
$(L-\lambda D)|_{D}$ is ample. Choose $\mu>\lambda$ such that $(L-\mu D)|_{D}$
is still ample. We claim that $s(mL,D)\ge\left\lfloor \lambda(m+1)\right\rfloor $
for sufficiently large and divisible $m$. To see this, let $m$ be
fixed and let $s=\left\lfloor \lambda(m+1)\right\rfloor $, then as
$m\gg0$ we have $s<\mu m$ and thus by Lemma \ref{lem:fujita}, $(mL-sD)|_{D}$
is very ample. Moreover, since $\lambda\not\in\mathbb{Q}$, we have
$s<\lambda(m+1)$, hence $mL-(s+1)D-(K_{X}+\Delta)\sim(m+1)L-sD$
is ample and as $X$ is globally $F$-regular, $H^{1}(X,mL-(s+1)D)=0$
by Lemma \ref{lem:vanishing}, thus $H^{0}(X,mL-sD)\rightarrow H^{0}(D,(mL-sD)|_{D})$
is onto. So $s(mL,D)\ge s=\left\lfloor \lambda(m+1)\right\rfloor $,
proving the claim. On the other hand, by Lemma \ref{lem:s(l,d)} we
have $\left\lfloor \lambda(m+1)\right\rfloor\le s(mL,D) \le\lambda m$
(for sufficiently divisible $m$). As $\lambda>0$ and $\lambda\not\in\mathbb{Q}$,
this is a contradiction.

Thus we have $\lambda\in\mathbb{Q}$. Let $M=L-\lambda D$. Under
either assumption of the lemma, $mM|_{D}$ is base point free for
sufficiently divisible $m$. We also have $H^{1}(X,mM-D)=0$ since
$X$ is globally $F$-regular and $mM-D-(K_{X}+\Delta)=mM+L$ is ample,
hence $H^{0}(X,mM)\rightarrow H^{0}(D,mM|_{D})$ is onto and the stable
base locus $B=\mathrm{Bs}(M)$ of $M$ is disjoint from $D$. On the
other hand, by Lemma \ref{lem:basept}, for any $x\in B$, there exists
a positive dimensional subvariety $C\subseteq X$ containing $x$
such that $M|_{C}$ is numerically trivial. By taking hyperplane sections
we may assume that $C$ is a curve. Clearly $C$ intersects $D$,
for otherwise $M|_{C}=L|_{C}$ is ample. Since $x\in B$ and $(M\cdot C)=0$,
we have $C\subseteq B$, but then $B\cap D$ contains $C\cap D$ and
in particular is nonempty, a contradiction. Thus $B=\emptyset$ and
$M$ is semiample.
\end{proof}

The next two lemmas are natural generalizations of \cite[Lemmas 4 and 7]{lz16} to pairs. We omit the proofs since the argument in \cite{lz16} works verbatim here.

\begin{lem}\label{lem:surfacecont}
Let $\pi:S\rightarrow T$ be a proper birational morphism between normal
surfaces and $\Delta$ an effective divisor on $S$. Let $C\subset S$ be a $K_{S}$-negative $\pi$-exceptional curve such that $C\not\subseteq\Supp(\Delta)$. Then $(-(K_{S}+\Delta)\cdot C)\le 1$, with equality if and only if $C$ is disjoint from $\Delta$ and $S$ has only Du Val singularities
along $C$.
\end{lem}

\begin{lem}\label{lem:birlocal}
Let $g:Y\rightarrow Z$ be a proper birational morphism between normal varieties and $\Delta$ an effective divisor on $Y$. Let $D$ be a smooth $g$-ample Cartier divisor on $Y$ such that $-(K_Y+\Delta+\lambda D)$ is $g$-nef for some $\lambda\ge1$. Assume that $Y$ is Cohen-Macaulay, $D\cap\Delta=\emptyset$ and $g|_{D}:D\rightarrow G=g(D)$ is an isomorphism, then $\lambda=1$, $\Ex(g)$ is disjoint from $\Delta$ and $Z$ is smooth along $G$.
\end{lem}

We are ready to finish the second part of the classification of varieties containing the projective space as a smooth divisor.

\begin{lem} \label{lem:rho2}
Let $(X,\Delta)$ be a pair and $D\cong\mathbb{P}^{n-1}$ a prime
divisor contained in the smooth locus of $X$ such that $L=-(K_{X}+\Delta+D)$
is ample. Assume that $\rho(X)\ge2$ and $\Delta\cap D=\emptyset$.
Then $X$ is isomorphic to a $\mathbb{P}^{1}$-bundle $\mathbb{P}(\mathcal{O}\oplus\mathcal{O}(-d))$
over $\mathbb{P}^{n-1}$ for some $d\in\mathbb{Z}_{\ge0}$ and $D$ is a section.
\end{lem}

\begin{proof}
By Lemma \ref{lem:F-reg} and our assumption, $(X,\Delta)$ is globally $F$-regular. Since $\rho(X)\ge2$, we may an ample divisor $H$ and $0<t\ll 1$ such that $(X,\Delta_1=\Delta+tH)$ is still globally $F$-regular, $L_1=-(K_X+\Delta_1+D)$ is ample and that $L_1$ and $D$ are linearly independent in $\mathrm{Pic}(X)_{\mathbb{Q}}$. Let $\lambda=\epsilon(L_1,D)$. Clearly $\lambda>0$. As $D\cong\mathbb{P}^{n-1}$, we have $\rho(X,D)=\rho(D)=1$. Thus
by Lemma \ref{lem:lambda}, $\lambda\in\mathbb{Q}$ and $M=L_1-\lambda D$
is semiample. Since $M\neq0$ in $\mathrm{Pic}(X)_{\mathbb{Q}}$,
it induces a morphism (with connected fibers) $g:X\rightarrow Y$
such that $\dim Y\ge1$ and $M=g^{*}H$ for some ample divisor $H$
on $Y$. We claim that $M|_{D}$ is ample. Indeed, if $(L_1-\lambda D)|_{D}=M|_{D}\sim_{\mathbb{Q}}0$,
then as $L_1$ is ample, $D|_{D}$ is ample as well. Let $S$ be a surface
in $X$ given by a complete intersection of general hyperplanes, then
we have $(D|_{S}^{2})>0$ and $(D|_{S}\cdot M|_{S})=0$, but then
by Hodge index theorem, $(M|_{S}^{2})<0$ ($M|_{S}$ is not numerically
trivial since $M$ is not), contradicting the fact that $M$ is nef.
Hence $M|_{D}$ is ample and by the same argument as in Lemma \ref{lem:lambda},
we know that $H^{0}(X,mM)\rightarrow H^{0}(D,mM|_{D})$ is onto, therefore
$g|_{D}$ is a closed embedding. Since $(X,\Delta)$ is globally $F$-regular, $X$ is Cohen-Macaulay by \cite[Theorem 1.18]{sz}. We also have $-(K_X+\Delta+(\lambda+1)D)\sim_{\bQ}M+tH$ which is $g$-ample. By Lemma \ref{lem:birlocal}, $g$ cannot be birational, hence induces an isomorphism $\mathbb{P}^{n-1}=D\cong Y$. If $C$ is a scheme theoretic fiber of $g$, then $C$ has dimension
one since $\dim(C\cap D)=0$. Since $g|_{D}$ is an isomorphism and
every component of $C$ intersects $D$, the curve $C$ is irreducible and reduced. Let $\cI_C$ be the ideal sheaf of $C$. Consider the exact sequence 
\[
\cdots\rightarrow R^1 g_*\cO_X \rightarrow H^1(C,\cO_C) \rightarrow R^2g_*\cI_C \rightarrow\cdots
\]
Since $g$ has fiber dimension at most one we have $R^2g_*\cI_C=0$ and as $X$ is globally $F$-regular we also have $R^1 g_*\cO_X=0$, thus $H^1(C,\cO_C)=0$ and since $C$ is integral this implies $C\cong\bP^1$. It follows that $g:X\rightarrow Y$ is a $\mathbb{P}^{1}$-fibration with a section $D$. Thus $X\cong\mathbb{P}_{Y}(\mathcal{O}_{Y}\oplus\mathcal{O}_{Y}(-d))$ for some $d\ge0$. 
\end{proof}

\subsection{Conic bundles} \label{sec:conic_bundle}

In this subsection we study conic bundles in positive characteristic. Later we will apply these results to classify varieties $X$ with $\epsilon(-K_X)=n$ and $\left((-K_X)^n\right)=n^n$.

\begin{defn} \label{def:Gor_cb}
Let $f:X\rightarrow Y$ be a proper morphism between normal quasi-projective
varieties. If the general fiber of $f$ is a plane conic (so is either a $\mathbb{P}^{1}$ or a double
line in characteristic $2$), we call $f$ a rational conic bundle. If $X$ is Cohen-Macaulay, every fiber of $f$ has pure dimension $1$, $f_{*}\mathcal{O}_{X}=\mathcal{O}_{Y}$, 
and there exists a Cartier divisor $D$ on $X$ such that $-K_{X}\equiv_{f}D$
is $f$-ample, then we call $f$ a Gorenstein conic bundle.
\end{defn}

\begin{lem} \label{lem:conic}
Let $C$ be a locally complete intersection \emph{(}l.c.i.\emph{)}
curve over $k$. Assume that $\omega_{C}^{-1}$ is ample. Then the
following are equivalent:
\begin{enumerate}
\item $h^{0}(C,\mathcal{O}_{C})=1$;
\item $\deg\omega_{C}=-2$ and every irreducible component of $C_{\mathrm{red}}$
is isomorphic to $\mathbb{P}^{1}$;
\item $C$ is a plane conic.
\end{enumerate}
\end{lem}

\begin{proof}
We will show $(1)\Rightarrow(2)\Rightarrow(3)\Rightarrow(1)$. By
Riemann-Roch and Serre duality (see \cite{liu} for Riemann-Roch
formula on singular curves) we have $\chi(\mathcal{O}_{C})=-\chi(\omega_{C})=-\deg\omega_{C}-\chi(\mathcal{O}_{C})$,
hence $-\deg\omega_{C}=2\chi(\mathcal{O}_{C})$. On the other hand,
since $\omega_{C}^{-1}$ is ample, $-\deg\omega_{C}>0$, thus if (1)
holds we have $0<\chi(\mathcal{O}_{C})=1-h^{1}(C,\mathcal{O}_{C})\le1$,
hence $\chi(\mathcal{O}_{C})=1$, $h^{1}(C,\mathcal{O}_{C})=0$ and
$\deg\omega_{C}=-2$. Moreover, as $\dim C=1$ the map $H^{1}(C,\mathcal{O}_{C})\rightarrow H^{1}(C_{i},\mathcal{O}_{C_{i}})$
is surjective for every component $C_{i}$ of $C_{\mathrm{red}}$,
which implies (2). 

Write $[C]=\sum a_{i}[C_{i}]$ as a 1-cycle where the $C_{i}$'s are
irreducible components of $C$, then $\deg\omega_{C}=\sum a_{i}\deg(\omega_{C}|_{C_{i}})$.
Since $\omega_{C}^{-1}$ is ample, $\deg(\omega_{C}|_{C_{i}})<0$.
Hence if (2) holds we have either $C$ is reduced with at most two
components or $[C]=2[C_{1}]$. If $C$ is reduced, the same Riemann-Roch
calculation as above yields $h^{1}(C,\mathcal{O}_{C})=0$, hence either
$C\cong\mathbb{P}^{1}$ or $C$ is the union $C_{1}\cup C_{2}$ of
two $\mathbb{P}^{1}$. In the latter case, by the exact sequence $0\rightarrow\mathcal{O}_{C}\rightarrow\mathcal{O}_{C_{1}}\oplus\mathcal{O}_{C_{2}}\rightarrow\mathcal{O}_{C_{1}\cap C_{2}}\rightarrow0$
we have $h^{0}(\mathcal{O}_{C_{1}\cap C_{2}})=1$ hence $C_{1}\cap C_{2}$
(scheme-theoretic intersection) consists of only one point and $C$
is a reducible conic. If $[C]=2[C_{1}]$, let $\mathcal{I}$ be the
ideal sheaf of $C_{1}$, then $\mathcal{I}^{2}=0$ and we have an
exact sequence $0\rightarrow\mathcal{I}\otimes\omega_{C}\rightarrow\omega_{C}\rightarrow\omega_{C}|_{C_{1}}\rightarrow0$.
As $\deg\omega_{C}=-2$ and $C_{1}\cong\mathbb{P}^{1}$, we get $\deg(\omega_{C}|_{C_{1}})=-1$
and $\chi(\omega_{C}|_{C_{1}})=0$. On the other hand, by Riemann-Roch
we have $\chi(\omega_{C})=\frac{1}{2}\deg\omega_{C}=-1$, hence $\chi(\mathcal{I}\otimes\omega_{C})=-1$
and $\deg\mathcal{I}=-1$. It follows that $C$ is an infinitesimal
extension (see \cite[II, Ex 8.7]{hartshorne}) of $C_{1}$ by $\mathcal{O}_{C_{1}}(-1)$,
which is classified by $H^{1}(C_{1},T_{C_{1}}(-1))=0$ by \cite[III, Ex 4.10]{hartshorne}. Since one such extension is given by the planar double line, it is isomorphic to $C$ and in particular $C$ is a plane conic. This proves (3). Finally it is clear that (3) implies (1).
\end{proof}

\begin{lem} \label{lem:conicbundle}
Let $f:X\rightarrow Y$ be a Gorenstein conic bundle. Assume that $Y$ is
smooth and $X$ is smooth at the generic points of every fiber of $f$, then $f$ is a conic
bundle.
\end{lem}

\begin{proof}
By dimension reason the singular locus of $X$ cannot dominate $Y$,
hence the general fiber $C$ of $f$ is l.c.i and $\omega_{C}^{-1}$
is ample by adjunction. Since $f_{*}\mathcal{O}_{X}=\mathcal{O}_{Y}$, $f$ has connected fibers and
we have $h^{0}(C,\mathcal{O}_{C})=1$, thus $C$ is a plane conic
by Lemma \ref{lem:conic}. Fix an arbitrary closed point $y\in Y$. By assumption, $X$ is smooth along the generic points of $C_y:=f^{-1}(y)$, and $f$ is flat by \cite[III. Ex 10.9]{hartshorne}. Since $(D\cdot C)=(-K_{X}\cdot C)=2$ (where $D$ is the Cartier
divisor on $X$ such that $-K_{X}\equiv_{f}D$ in the definition of Gorenstein conic bundles), $C_y$ has
at most $2$ irreducible components (counting multiplicities). In particular, at a general point $x$ of every component of $C_y$ we have $\dim_{k}\mathfrak{m}_{C_y,x}/\mathfrak{m}_{C_y,x}^{2}\le2$ and the image of $\mathfrak{m}_{Y,y}/\mathfrak{m}_{Y,y}^{2}\rightarrow\mathfrak{m}_{X,x}/\mathfrak{m}_{X,x}^{2}$
has dimension at least $n-2$ (where $n=\dim X$). It follows that for a general curve $B\subseteq Y$ passing through $y$, the scheme-theoretic preimage $S=f^{-1}(B)$ is smooth at the generic points of $C_y$. Hence $S$ is
generically reduced; it is also $S_{2}$ since $X$ is Cohen-Macaulay. Therefore $S$
is reduced.

We divide into cases to show that $C_y$ is a plane conic. If the general fiber $C$ is reduced, then it is a smooth rational curve, hence $S$ is smooth in codimension one (as it is smooth both outside $C_y$ by Bertini's theorem and also at the generic points of $C_y$) and thus normal. By adjunction $S\rightarrow B$
is also a Gorenstein conic bundle, so by \cite[Lemma 15]{lz16} (whose proof works in any characteristic), $C_y=f|_{S}^{-1}(y)$
is a plane conic. If $C$ is a double line (which only happens in
characteristic 2), then we have $C_y=2C_{1}$ as a 1-cycle.
Let $\bar{S}\rightarrow S$ be the normalization of $S$, $\Delta\subseteq\bar{S}$
the conductor and $g:\bar{S}\rightarrow\bar{B}$ the Stein factorization
of $\bar{S}\rightarrow B$. Let $\bar{C}_{1}$ be the strict transform
of $C_{1}$. Then the general fiber of $g$ is a smooth rational curve,
therefore $\bar{B}\rightarrow B$ is purely inseparable of degree
2 and indeed every fiber of $g$ is irreducible and reduced. By \cite[II.2.8]{kol96}, $g$ is a $\mathbb{P}^{1}$-bundle. It follows that $2=(-K_{\bar{S}}\cdot\bar{C}_{1})=(-K_{S}\cdot C_{1})+(\Delta\cdot\bar{C}_{1})$,
but since $(-K_{S}\cdot C_{1})=(D\cdot C_{1})=1$, we get $(\Delta\cdot\bar{C}_{1})=1$.
Hence the conductor intersects $\bar{C}_{1}$ transversally at a single
point and $\bar{C}_{1}\rightarrow C_{1}$ is an isomorphism. In particular,
$C_{1}\cong\mathbb{P}^{1}$. Note that $\chi(\mathcal{O}_{C_y}(-D))=\chi(\mathcal{O}_{C}(-D))=-1$,
by the exact sequence 
\[
0\rightarrow\mathcal{I}_{C_{1}}(-D)\rightarrow\mathcal{O}_{C_y}(-D)\rightarrow\mathcal{O}_{C_{1}}(-1)\rightarrow0
\]
and the similar proof of $(2)\Rightarrow(3)$ in Lemma \ref{lem:conic}
we see that $C_y$ is a planar double line.

We therefore conclude that in all cases $C_y$ is a plane conic.
As $C_y$ is cut out by hypersurfaces, $X$ has only hypersurface
singularities and in particular is Gorenstein. The lemma now follows
from standard argument (i.e. $\mathcal{E}=f_{*}\omega_{X}^{-1}$
is a vector bundle of rank 3 on $Y$ and $X$
embeds into $\mathbb{P}(\mathcal{E})$, see e.g. \cite{conic}).
\end{proof}

The following corollary is well-known in characteristic zero by the work of \cite{ando}.

\begin{cor} \label{cor:ratconic}
Let $f:X\rightarrow Y$ be a proper morphism. Assume that every fiber
of $f$ has dimension $1$, $-K_{X}$ is $f$-ample and $f_{*}\mathcal{O}_{X}=\mathcal{O}_{Y}$.
Then $f$ is a rational conic bundle. If in addition $X$ and $Y$
are both smooth, then $f$ is a conic bundle.
\end{cor}

\begin{proof}
This is an immediate consequence of the above lemma.
\end{proof}

The next lemma is essentially \cite[Lemma 17]{lz16}, with strong $F$-regularity in place of klt singularity.

\begin{lem} \label{lem:basechange}
Let $f:X\rightarrow Y$ be a Gorenstein conic bundle and $\phi:\tY\rightarrow Y$ a finite separable morphism. Let $\tX$ be the normalization of $X\times_{Y}\tY$. Assume that $X$ is smooth at the generic points of every fiber of $f$ and strongly $F$-regular at all points, and the branch divisor of $\phi$ is disjoint from the singular locus of $\tY$ and $Y$. Then $\tf:\tX\rightarrow\tY$ is also a Gorenstein conic bundle.
\end{lem}

\begin{proof}
By shrinking $Y$ we may assume either $\phi$ is \'etale in codimension one or both $Y$ and $\tY$ are smooth. In the first case $\tX$ is also strongly $F$-regular by \cite[Theorem 2.7]{wat} hence is Cohen-Macaulay by \cite[Theorem 1.18]{sz}, and the other properties of Gorenstein conic bundles are preserved by a finite base change that is \'etale in codimension one. In the second case $f$ is a conic bundle by Lemma \ref{lem:conicbundle}, hence the same holds for $\tf$.
\end{proof}

\subsection{Proof of Theorem \ref{thm:Pn}, \ref{thm:=n} and \ref{thm:n^n}} \label{sec:pf-1}

Before proving these theorems, we make a few reductions and fix the following notations. After a base change, we first assume that the base field $k$ is uncountable. Since the Seshadri constant of a line bundle $L$ attains its maximum at a very general point of $X$, we may also assume that $x\not\in\Supp(\Delta)$. Let $\sigma:Y\rightarrow X$ be the blow up of $X$ at $x$ and let $E$ be the exceptional divisor. Let $\Delta$ also denote its strict transform on $Y$.

\begin{proof}[Proof of Theorem \ref{thm:Pn}(1)]
As $\epsilon(L,x)>n$, $-(K_Y+\Delta+E)=\sigma^*L-nE$ is ample. Clearly $\rho(Y)\ge 2$ and $\Delta\cap E=\emptyset$, thus by Lemma \ref{lem:rho2}, $Y$ is isomorphic to a $\mathbb{P}^{1}$-bundle $\mathbb{P}(\mathcal{O}\oplus\mathcal{O}(-d))$
over $\mathbb{P}^{n-1}$ for some $d\in\mathbb{Z}_{\ge0}$ and $E$ is a section. But since $\mathcal{N}_{E/Y}\cong\cO_E(-1)$, we have $d=1$ and $E$ is the unique negative section. It follows that $Y$ is the blowup of $\bP^n$ at a point and therefore $X\cong\bP^n$.
\end{proof}

\begin{proof}[Proof of Theorem \ref{thm:=n} when $(L^n)>n^n$]
By assumption, $D=-(K_Y+\Delta+E)=\sigma^*L-nE$ is nef and big and $(Y,\Delta)$ is globally $F$-regular by Corollary \ref{cor:=n->F-reg}. We claim that $D$ is semiample. Note that $mD-E-(K_Y+\Delta)=(m+1)D$ is nef and big, so by Lemma \ref{lem:vanishing}, $H^1(Y, \cO_Y(mD-E))=0$ for all $m\ge0$ and 
\begin{equation} \label{eq:surj}
H^0(Y,\cO_Y(mD))\rightarrow H^0(E,\cO_E(mD))
\end{equation}
is surjective, hence $E$ is disjoint from the stable base locus $\mathrm{Bs}(D)$ of $D$. Now if $y\in \mathrm{Bs}(D)$, then by Lemma \ref{lem:basept} there exists a curve $C$ containing $y$ such that $(D\cdot C)=0$, but then $C\subseteq\mathrm{Bs}(D)$ and since $-(K_Y+\Delta)$ is ample, we have $(E\cdot C)>0$ and in particular $E\cap\mathrm{Bs}(D)\neq\emptyset$, a contradiction. Hence $\mathrm{Bs}(D)=\emptyset$.

Therefore, $D$ is semiample and induces a birational morphism $g:Y\rightarrow Z$. Clearly $\Delta\cap E=\emptyset$. Since $D$ is not ample (otherwise $\epsilon(L,x)>n$), $g$ is not the identity morphism. By the surjectivity of (\ref{eq:surj}), $g|_E$ is a closed embedding as $D|_E$ is ample. Note that $-(K_Y+\Delta+E)\sim_{g.\bQ.}0$, so by Lemma \ref{lem:birlocal}, $g$ is an isomorphism around $\Delta$ and $Z$ is smooth along $G=g(E)\cong\bP^{n-1}$. It follows that $\Delta$ is also disjoint from $G$ (here we identify $\Delta$ with its image in $Z$). By the construction of $g$, $-(K_Z+\Delta+G)$ is ample. By Lemma \ref{lem:rho1} and \ref{lem:rho2}, one of the following holds:

\begin{enumerate}
\item $Z\cong\mathbb{P}(1^{n},d)$ for some $d\in\mathbb{Z}_{>0}$ and $G$ is the hyperplane defined by the vanishing of the last coordinate;
\item $Z$ is isomorphic to a $\mathbb{P}^{1}$-bundle $\mathbb{P}(\mathcal{O}\oplus\mathcal{O}(-d))$
over $\mathbb{P}^{n-1}$ for some $d\in\mathbb{Z}_{\ge0}$ and $G$ is a section; or
\item $n=2$, $Z\cong\mathbb{P}^{2}$ and $G$ is a smooth conic.
\end{enumerate}

We now show that $Y$ is the blowup of $Z$ along a hypersurface in $G\cong\bP^{n-1}$. This essentially follows from the argument of \cite[Lemma 11]{lz16}, once we have the vanishing
\begin{equation} \label{eq:R^1=0}
R^{1}g_{*}\mathcal{O}_Y(-mW)=0
\end{equation}
for all $m\ge0$, where $W=g^{*}G-E$. But since $-mW-(K_Y+\Delta)\sim_{g.\mathbb{Q}}(m+1)E$ is $g$-ample, (\ref{eq:R^1=0}) follows from Corollary \ref{cor:relvanish} and the global $F$-regularity of $(Y,\Delta)$. We also notice that $G$ is nef by \cite[Lemma 10]{lz16}. We can therefore apply the same argument of \cite[Lemma 13]{lz16} to conclude the proof.
\end{proof}

\begin{proof}[Proof of Theorem \ref{thm:=n} and \ref{thm:n^n}]
The proof of both theorems is intertwined and a bit lengthy, so we divide it into several steps. Let $D=\sigma^{*}L-nE=-(K_{Y}+\Delta+E)$. Since the case $(D^n)=(L^n)-n^n>0$ of Theorem \ref{thm:=n} is already treated, we may assume $(D^{n})=0$.

\medskip
\emph{Step 1} ($D$ is semiample). By assumption $-(K_{Y}+\Delta)$ is ample, $D$ is nef and $(D^{n})=0$, hence
by Riemann-Roch we have 

\begin{eqnarray*}
h^{0}(Y,mD) & \ge & h^0(Y,mL)-h^0(E,\cO_{mnE}) \\
            & =   & \frac{(D^{n})}{n!}m^{n}+\frac{(-K_{Y}\cdot D^{n-1})}{2(n-1)!}m^{n-1}+O(m^{n-2})\\
            & =   & \frac{n^{n-1}}{2(n-1)!}m^{n-1}+O(m^{n-2})
\end{eqnarray*}

It follows that $\nu(Y,D)=\kappa(Y,D)=n-1$ where $\nu(Y,D)=\max\{d\,|\,D^{d}\not\equiv0\}$
is the numerical dimension of $D$. By \cite[Proposition 2.1]{kaw}, there exists a diagram of normal varieties (the characteristic zero assumption in \cite{kaw} is only used to make the varieties in the diagram smooth)
\[
\xymatrix{Y_0\ar[r]^{\mu}\ar[d]_{f} & Y\\
Z_{0}
}
\]
and a nef and big divisor $D_{0}$ on $Z_{0}$ such that $\mu$ is
birational, $f$ is equi-dimensional and $\mu^{*}D=f^{*}D_{0}$. It
follows that for every closed point $y\in Y$ there exists a curve
$C_{y}\subseteq Y$ (coming from a fiber of $f$ that intersects $\mu^{-1}(y)$)
such that $(D\cdot C_{y})=0$ and $C_{y}$ is unique if $y$ is general.
Since $\kappa(Y,D)=n-1$, for sufficiently divisible $m$ the linear
system $|mD|$ gives a rational map $g:Y\dashrightarrow Z$ with $\dim Z=n-1$.
As $(D\cdot C_{y})=0$, $g$ is defined along $C_{y}$ if $y\not\in\mathrm{Bs}(D)$,
hence $C_{y}$ is the (at least set-theoretic) general fiber of $g$
and we get a proper morphism $g_{1}:Y_{1}\rightarrow Z_{1}$ with $g_{1*}\mathcal{O}_{Y_{1}}=\mathcal{O}_{Z_{1}}$
by shrinking $Z$ and taking Stein factorization. Let $Z_{1}\dashrightarrow\mathrm{Chow}(Y)$
be the rational map induced by $g_{1}$ (see \cite[I.3-4]{kol96}
for the definition and basic properties of Chow varieties) and let $Z'$
be the normalization of the closure of the image of $Z_{1}$ in $\mathrm{Chow}(Y)$. By Corollary \ref{cor:ratconic}, the general fiber $L_z$ of $g_1$ is a plane conic. 

Assume for the moment that $L_z$ is a smooth conic (this is automatically satisfied when $p=\mathrm{char}(k)\neq2$ or $(\Delta\cdot C_y)>0$: in the latter case, $(D\cdot C_y)<(-K_Y-E\cdot C_y)\le 0$ if $L_z$ is nonreduced). In particular, $g_{1}$ has reduced
general fiber and we get a universal family $q:U\rightarrow Z'$.
Let $u:U\rightarrow Y$ be the cycle map. We claim that $u$ is an
isomorphism.

To this end let $C\subseteq U$ be a curve that is contracted by $u$.
Since $u$ is injective on every fiber of $q$, $q(C)$ is not a point.
Let $S$ be the irreducible component of $q^{-1}(q(C))$ that contains
$C$. By construction $(u^{*}D\cdot C)=(u^{*}D\cdot F)=0$ where $F$
is any component in a fiber of $q$, thus by \cite[Proposition 2.5]{nefreduction}, $u^{*}D|_{S}$ is numerically trivial. Let $T=u(S)$, then $T$ is a surface in $Y$ such that $D|_{T}\equiv0$.
As $D=-K_{Y}-\Delta-E$ and $-(K_{Y}+\Delta)$ is ample, $T$ must intersect $E$
and $\dim(T\cap E)\ge1$, but then since $D|_{E}\sim-nE|_{E}$ is
ample, $D|_{T\cap E}$ cannot be numerically trivial, a contradiction.
Hence $u$ is quasi-finite and is indeed an isomorphism since it is
also birational and $Y$ is normal.

Thus we get a rational conic bundle $Y\cong U\rightarrow Z'$ with
general fiber $C_{y}$. As $(D\cdot C_{y})=0$, any $G\in|mD|$ can not dominate $Z'$, thus as every fiber of $q$ has pure dimension 1, $G$ is in fact the pullback of an effective divisor on $Z'$. On the other hand by \cite[Lemma 5.16]{km98} applied to the finite morphism $E\cong\mathbb{P}^{n-1}\rightarrow Z'$, we see that $Z'$ is $\mathbb{Q}$-factorial of Picard number one. Hence $D$ is semiample if $L_z$ is reduced.

\medskip
\emph{Step 2} (Proof of Theorem \ref{thm:=n} when $\Delta\neq0$). We claim that $\Delta$ is $\bQ$-Cartier. Using the notation and construction in Step 1, there are three cases to consider.

Suppose first that $(\Delta\cdot C_y)>0$. Then $|mD|(m\gg0)$ defines a morphism $g:Y\rightarrow Z$ by Step 1. Moreover $(E\cdot C_{y})=(-(K_{Y}+\Delta)\cdot C_{y})<2$, thus $(E\cdot C_y)=1$ and $E\rightarrow Z$ is an isomorphism. Since $E$ intersects every component in the fiber of $g$ (for otherwise this component would have zero intersection number with the ample divisor $-(K_Y+\Delta)$), we see that every fiber of $g$ is generically irreducible and reduced. By \cite[Lemma 6]{lz16}, there exists a codimension $\ge2$ subset $W\subseteq Z$ such that $Y\backslash g^{-1}(W)$ is isomorphic to a $\bP^1$-bundle over $Z\backslash W$. It follows that the class group of $Y$ is generated by $E$ and $g^*\Pic(Z)$ and in particular $Y$ is $\bQ$-factorial. Thus $\Delta$ is $\bQ$-Cartier in this case.

Assume next that $(\Delta\cdot C_y)=0$ and $L_z$ is a smooth conic. Again we have a rational conic bundle $g:Y\rightarrow Z$ defined by $|mD|(m\gg0)$. Since $(\Delta\cdot C_y)=0$, $\Delta_Z=g(\Delta)$ is a divisor in $Z$. As $(E\cdot C_{y})=(-(K_{Y}+\Delta)\cdot C_{y})=2$, every fiber of $g$ has at most 2 components (counting multiplicity), thus by the same proof of \cite[Lemma 16]{lz16}, $g^{-1}(u)$ is a plane conic where $u$ is a generic point of $\Delta_Z$. We claim that $\Delta$ is proportional to $g^*\Delta_Z$ over $u$. Suppose not, then $g^{-1}(u)$ is not irreducible and there exists a component $F$ of $g^{-1}(u)$ such that $(\Delta\cdot F)>0$. But we also have $(-K_Y\cdot F)=1\le (E\cdot F)$, hence $(D\cdot F)=(-(K_Y+\Delta+E)\cdot F)<0$, a contradiction. Therefore, we can find a $\bQ$-divisor $\Delta_1$ supported on $\Delta_Z$ such that $\Delta=g^*\Delta_1$. Recall that $Z$ is $\bQ$-factorial, thus $\Delta$ is also $\bQ$-Cartier in this case.

Finally suppose that $(\Delta\cdot C_y)=0$ and $L_z$ is a nonreduced conic. In particular $p=2$. Taking the base change of $g_1:Y_1\rightarrow Z_1$ by $E_1\rightarrow Z_1$ (where $E_1$ is the preimage of $Z_1$ in $E$), we get a family $h_1:U_1\rightarrow E_1$ of reduced curves in $Y$ with general member $C_y$. As in Step 1 we may extend $h_1$ to a universal family $h:U\rightarrow V$ (where $V$ is the closure of the image of $E_1$ in $\mathrm{Chow}(Y)$) and the same argument there implies that the cycle map $u:U\rightarrow Y$ is quasi-finite, thus is an inseparable double cover. It follows that the Frobenius map of $Y$ factors through $u$, hence $u^{-1}(E)$ is $\bQ$-factorial and $\Delta$ is $\bQ$-Cartier if and only if $u^*\Delta$ is $\bQ$-Cartier. But as $(E\cdot C_y)=(-K_Y\cdot C_y)=1$, every fiber of $h$ is generically integral, thus $u^*\Delta$ is the pullback of a divisor from $V$. Since $V$ is dominated by $u^{-1}(E)$, it is $\bQ$-factorial by \cite[Lemma 5.16]{km98}. Hence $\Delta$ is $\bQ$-Cartier in this last case.

Now that $\Delta\neq0$ is $\bQ$-Cartier, we may replace $(X,\Delta)$ by $(X,(1-c)\Delta)$ for $0<c\ll 1$ and reduce to the case $(L^n)>n^n$ using \cite[Theorem B]{fkl}. This finishes the proof of Theorem \ref{thm:=n}.

\medskip
In the remaining part of the proof, we assume that $\Delta=0$ and $p>2$. By Step 1, this implies that $D$ is semiample and induces a morphism $g:Y\rightarrow Z$. We have $-K_Y\sim_{g.\bQ.}E$, thus $g$ is a Gorenstein conic bundle if $Y$ is Cohen-Macaulay. As $Y$ is smooth along $E$ and $E$ is $g$-ample, $Y$ is smooth at the generic points of every fiber of $g$.

\medskip
\emph{Step 3} (Surface case). If $Y$ is a surface, then by \cite[Lemma 15]{lz16}, $Y$ has only Du Val singularity. It follows that $X$ is a Gorenstein log del Pezzo surface of degree $(K_{X}^{2})=4$. Hence from now on, we assume that $n=\dim X\ge3$.

\medskip
\emph{Step 4} ($Y$ is globally $F$-regular). It is clear that $E$ is a double section of $g$. We may
assume that $E\rightarrow Z$ is ramified (the quasi-\'etale case
is similar and even simpler), then $Z\cong\mathbb{P}(1^{n-1},2)$,
$g|_{E}$ is ramified along the hyperplane $M\subseteq Z$ defined
by the vanishing of the last coordinate and $K_{E}=g^{*}(K_{Z}+\frac{1}{2}M)$.
Since the general fiber of $g$ is a smooth rational curve, we can
choose an ample Cartier divisor $H$ on $Z$ such that $Y\backslash g^{-1}H$
is smooth and globally $F$-regular. We then have a similar diagram as
in the proof of Lemma \ref{lem:F-reg}:
\[
\xymatrix{
 F_{*}^{e}\mathcal{O}_{X}((1-p^{e})(K_{X}+E)-g^*H)\ar[r]^-{j}\ar[d]^{\mathrm{Tr}_{X}^{e}} & F_{*}^{e}\mathcal{O}_{E}((1-p^{e})K_{E}-g|^*_E H)\ar[d]^{\mathrm{Tr}_{E}^{e}}\\
 \mathcal{O}_{X}\ar[r] & \mathcal{O}_{E}
}
\]
Although $H^{0}(j)$ is not surjective, its image contains $g^{*}H^{0}(Z,(1-p^{e})(K_{Z}+\frac{1}{2}M)-H)$
for $e\gg0$, hence by the same argument in Lemma \ref{lem:F-reg},
it suffices to show that $H^{0}(\mathrm{Tr}_{Z}^{e})$ is surjective
for $e\gg0$ where $\mathrm{Tr}_{Z}^{e}:F_{*}^{e}\mathcal{O}_{Z}((1-p^{e})(K_{Z}+\frac{1}{2}M)-H)\rightarrow\mathcal{O}_{Z}$
is the trace map. But it is clear that the toric pair $(Z,\frac{1}{2}M)$
is globally $F$-regular, so we are done.

\medskip
\emph{Step 5} (Analysis of the Gorenstein conic bundles). Since $Y$ is globally $F$-regular, it is Cohen-Macaulay by \cite[Theorem 1.18]{sz}, thus $g$ is a Gorenstein conic bundle and is indeed a conic bundle over the smooth locus of $Z$ by Lemma \ref{lem:conicbundle}. Let $W$ be the normalization of $Y\times_{Z}E$, then since $E\cong\mathbb{P}^{n-1}\rightarrow Z$
is quasi-\'etale unless $Z\cong\mathbb{P}(1^{n-1},2)$ in which case
the branch divisor $G$ is disjoint from Sing$(Z)$, $W\rightarrow E\cong\mathbb{P}^{n-1}$
is a also Gorenstein conic bundle by Lemma \ref{lem:basechange}, and
is indeed a conic bundle away from the preimage of Sing$(Z)$. Away from the branch divisor $G$, the map $E\to Z$ is quasi-\'etale and hence the preimage of $E$ in $W$ splits into two disjoint components that remain Cartier and smooth. It follows that $W$ is smooth along the preimage of $E$ (in particular, $W$ is smooth at generic points of every fiber of $g$ outside $G$) and therefore by Lemma \ref{lem:conicbundle}, $g$ is a conic bundle away from $G$. Hence $g$ is a conic bundle everywhere. The theorem now follows from the same calculation as in the proof of \cite[Lemma 19]{lz16}.
\end{proof}

\begin{proof}[Proof of Theorem \ref{thm:Pn}(2)]
The surface case is just \cite[Example 5.5]{hara} while the higher dimensional case follows from Corollary \ref{cor:=n->F-reg} and Step 4 of the previous proof.
\end{proof}

\section{Weak boundedness} \label{sec:weakbdd}

We now turn to the proof of Theorem \ref{thm:weakbdd_m}. Recall that by Lemma \ref{lem:reduction}, we may assume that the variety $X$ admits a Mori fiber space structure.

\subsection{Fiber type case} \label{sec:fiber type}

We first look at the case when the Mori fiber space $f:X\rightarrow Y$ is of fiber type, which will be assumed for all Mori fiber spaces is this subsection. Consider the following assumptions on a variety $X$ of dimension $n$:

\begin{ass} \label{ass:raylength}
There exists a constant $A=A(n,\epsilon)>0$ that only depends on $n$ and $\epsilon$ such that every $K_X$-\emph{positive} extremal ray of $X$ is generated by a curve $C$ with $(K_X\cdot C)<A$.
\end{ass}

\begin{ass} \label{ass:seshadri_bound}
$\epsilon_{m}(-K_{X})>n-1+\epsilon$.
\end{ass}

Our plan in this section is to show that the set of Mori fiber spaces $f:X\rightarrow Y$ that satisfy these two assumptions is weakly bounded. Using suitable cone theorem for threefold pairs in positive characteristics, we find that Assumption \ref{ass:raylength} is implied by Assumption \ref{ass:seshadri_bound} in dimension 3, thus proving Theorem \ref{thm:weakbdd_m} in the case of Mori fiber spaces.

We first deduce some direct consequences of the Assumptions \ref{ass:raylength} and \ref{ass:seshadri_bound}.

\begin{lem} \label{lem:F and Y}
Let $f:X\rightarrow Y$ be a Mori fiber space such that $X$ satisfies Assumption \ref{ass:seshadri_bound}, then the general fiber $F$ of $f$ is reduced and isomorphic to $\bP^{n-1}$, $\dim Y=1$ and $\rho(X)=2$.
\end{lem}

\begin{proof}
Let $\tF$ be the normalization of the reduced subscheme of $F$. By \cite[Theorem 1.1]{general_fiber}, there exists an effective Weil divisor $D$ on $\tF$ such that $K_{\tF}+D\sim K_X|_{\tF}$. Since $f:X\rightarrow Y$ is a Mori fiber space, $-K_X|_{\tF}$ is ample, hence by Assumption \ref{ass:seshadri_bound} and Lemma \ref{lem:restrict_m}, we have $\epsilon(-K_{\tF}-D)\ge\epsilon_m(-K_X)>n-1\ge \dim \tF$. By Theorem \ref{thm:Pn}, this implies that $\tF\cong\bP^{n-1}$, $\dim Y=1$ and $D=0$ (note that components of $D$ have integral coefficients). But then by \cite[Theorem 7.1, Lemma 7.2]{Badescu}, $F$ itself is reduced and we also have $(K_X+F)|_{\tF}=K_{\tF}$. Since $\tF$ is smooth, $F$ is normal by \cite[Theorem A]{das}, thus $F\cong\tF\cong\bP^{n-1}$. Finally, since $f$ is a Mori fiber space, we get $\rho(X)=\rho(X/Y)+\rho(Y)=2$.
\end{proof}

\begin{lem} \label{lem:f-reg along F}
Let $f:X\rightarrow Y$ be a Mori fiber space with general fiber $F$. Assume that $X$ satisfies Assumption \ref{ass:seshadri_bound}, then there exists $D\sim_\bQ -\frac{1}{1-\epsilon}K_X$ such that $(X,D)$ is strongly $F$-regular along $F$.
\end{lem}

\begin{proof}
This follows from Lemma \ref{lem:mult and test ideal} and the same proof of \cite[Corollary 5.2]{me}.
\end{proof}

\begin{lem} \label{lem:explicit_ample}
Let $f:X\rightarrow Y$ be a Mori fiber space with general fiber $F$. Assume that $X$ satisfies both Assumptions \ref{ass:raylength} and \ref{ass:seshadri_bound}. Then $-K_X+AF$ is ample. 
\end{lem}

\begin{proof}
We may assume that $-K_X$ is not nef. By Lemma \ref{lem:F and Y}, $F$ is smooth divisor in $X$, $Y$ is a curve and $\rho(X)=2$. It follows that the Mori cone $\overline{NE}(X)$ is 2-dimensional, hence by Assumption \ref{ass:raylength}, it is generated by a curve $C$ that dominates $Y$ with $0<(K_X\cdot C)<A$ and a line $l$ in $F$. Let $L=-K_X+AF$, then since $F$ is Cartier we have $(L\cdot C)\ge (-K_X\cdot C)+A>0$ and it is clear that $(L\cdot l)>0$. Hence $L$ lies in the interior of the nef cone and is therefore ample.
\end{proof}

We will also need the following result:

\begin{lem} \label{lem:Qeff}
Let $\phi:X\rightarrow Y$ be a projective morphism onto a curve $Y$ with general fiber $F$. Let $L$ be a $\bQ$-Cartier big divisor on $X$ and let $\lambda<\frac{\vol_X(L)}{n\cdot \vol_F(L|_F)}$ be a rational number where $n=\dim X$, then $L-\lambda F$ is $\mathbb{Q}$-effective.
\end{lem}

\begin{proof}
Let $m$ be a sufficiently divisible positive integer and let $k=m\lambda$.
We may assume $\lambda>0$ and $k\in\mathbb{Z}$. By the exact sequence
\[
0\rightarrow\mathcal{O}_{X}(mL-lF)\rightarrow\mathcal{O}_{X}(mL-(l-1)F)\rightarrow\mathcal{O}_{F}(mL)\rightarrow0
\]
for $l=1,\cdots,k$ we get 
\begin{eqnarray*}
h^{0}(X,\mathcal{O}_{X}(mL-kF)) & \ge & h^{0}(X,\mathcal{O}_{X}(mL))-k\cdot h^{0}(F,\mathcal{O}_{F}(mL)) \\
                                & =   & \left( \frac{\vol_X(L)}{n!}-\lambda\frac{\vol_F(L|_F)}{(n-1)!} \right)\cdot m^{n}+O(m^{n-1})>0, 
\end{eqnarray*}
where the middle equality is the asymptotic Riemann-Roch formula and the last inequality follows from the assumption on $\lambda$. Since $mL-kF=m(L-\lambda F)$, the lemma follows.
\end{proof}

\begin{rem}
When $F$ is normal, the above lemma also holds even if $L$ is not $\bQ$-Cartier since in this case the restriction $L|_F$ is well defined (it is determined by the restriction of $L$ to the smooth locus of $F$). Also note that if in addition $L$ has integral coefficients, then in the above argument $m$ can be taken as any sufficiently large integer such that $m\lambda\in\bZ$. Therefore, if the denominator of $\lambda$ is not divisible by $p$, then we can choose $\Delta\sim_\bQ L-\lambda F$ such that $m\Delta$ has integral coefficient for some integer $m$ that is not divisible by $p$.
\end{rem}

Recall that in characteristic zero, the weak boundedness of varieties with large moving Seshadri constants \cite{me} is proved using the connectedness lemma of Koll\'ar-Shokurov, which is not yet available in positive characteristic. The following weaker version, however, suffices for the purpose of this section.

\begin{lem} \label{lem:weak connectedness}
Let $(X,D)$ be a pair such that $X$ is projective and $-(K_X+D)$ is ample, then $(X,D)$ has at most one {\it good} $F$-pure center.
\end{lem}

Let us elaborate the meaning of {\it good $F$-pure center} here. Let $(X,D)$ be a pair and $Y\subseteq X$ a normal subvariety such that the following conditions hold in a neighbourhood of $Y$ (we refer to \cite{schwede-adjunction} for the definition of center of $F$-purity):
\begin{enumerate}
    \item The Cartier index of $K_X+D$ is not divisible by $p$;
    \item $(X,D)$ is sharply $F$-pure along $Y$ and $Y$ is a center of sharp $F$-purity for $(X,D)$.
\end{enumerate}
Then by the main theorem of \cite{schwede-adjunction}, there exists a canonically determined effective divisor $D_Y$ such that $(K_X+D)|_Y\sim_\bQ K_Y+D_Y$. We say that $Y$ is a good $F$-pure center of $(X,D)$ if $(Y,D_Y)$ is globally $F$-regular.

\begin{proof}
Suppose there are two distinct good $F$-pure centers $W_1$, $W_2$ of the pair $(X,D)$ and we will derive a contradiction. By the main theorem of \cite{schwede-adjunction}, both $W_i$ are minimal among centers of sharp $F$-purity for $(X,D)$, hence by \cite[Lemma 3.5]{schwede-F-center}, $W_1$ is disjoint from $W_2$. Let $W=W_1\cup W_2$ and let $e>0$ be a sufficiently divisible integer. We have the following commutative diagram
\[
\xymatrix{F_{*}^{e}\mathcal{O}_{X}((1-p^{e})(K_{X}+D))\ar[r]^{\psi}\ar[d]^{\mathrm{Tr}_{X}^{e}} & F_{*}^{e}\mathcal{O}_{W}((1-p^{e})(K_{W}+D_{W}))\ar[d]^{\mathrm{Tr}_{W}^{e}}\\
\mathcal{O}_{X}\ar[r]^{\phi} & \mathcal{O}_{W}.
}
\]
Since the $W_i$'s are globally $F$-regular, $H^0(\mathrm{Tr}_{W}^{e})$ is surjective. On the other hand, the cokernel of $H^0(\psi)$ lies in $H^1(X,\mathcal{I}_{W}((1-p^{e})(K_{X}+D))$, which vanishes for $e\gg 0$ since $-(K_X+D)$ is ample. Hence $\phi\circ\mathrm{Tr}_{X}^{e}=\mathrm{Tr}_{W}^{e}\circ\psi$ induces a surjection on $H^0$. In particular, the natural restriction $H^0(X,\mathcal{O}_{X})\rightarrow H^0(W,\mathcal{O}_{W})$ is surjective. But as $W$ contains two connected components, this is a contradiction.
\end{proof}

We now prove that the Assumptions \ref{ass:raylength} and \ref{ass:seshadri_bound} together imply weak boundedness for Mori fiber spaces.

\begin{thm} \label{thm:mfs}
Given $v,\alpha>0$, there exists a constant $M=M(n,A,v,\alpha)$ depending only on $n$, $A$, $v$ and $\alpha$ such that if $f:X\rightarrow Y$ is a Mori fiber space such that $X$ satisfies Assumption \ref{ass:raylength}, $Y$ is a curve and the general fiber $F$ is globally $F$-regular with $\vol(-K_F)<v$ and  $\mathrm{fpt}(F)>\alpha$, then $\vol(-K_X)<M$.
\end{thm}

\begin{proof}
We may assume $\alpha<1$. Let $0<\lambda<(nv)^{-1}\vol(-K_X)$, $0<r<\alpha$ be rational numbers whose denominators are not divisible by $p$. Apply Lemma \ref{lem:Qeff} and its subsequent remark to $L=-K_{X}$ we see that there exists an effective divisor $\Delta\sim_{\mathbb{Q}}-K_{X}-\lambda F$ such that $m\Delta$ has integral coefficients for some $p\nmid m$. Let $D=F_1+F_2+r\Delta$ where $F_1$ and $F_2$ are two distinct general fibers of $f$ . We have $-(K_{X}+D)\sim_{\mathbb{Q}}-(1-r)K_{X}+(r\lambda-2)F$. Suppose that $r\lambda-2\ge A(1-r)$ where $A$ is the constant in Assumption \ref{ass:raylength}, then $-(K_{X}+D)$ is ample by Lemma \ref{lem:explicit_ample}. Perturbing $r$, we may assume that $(1-r)lK_X$ is Cartier for some $p\nmid l$. It follows that the Cartier index of $K_X+D$ is not divisible by $p$. On the other hand, as $(K_X+D)|_{F_i}\sim K_{F_i}+D_{F_i}$ where $D_{F_i}\sim_{\bQ}-rK_{F_i}$ and $r<\mathrm{fpt}(F)$, we see that $(F_i,D_{F_i})$ is globally $F$-regular. By \cite[Theorem A]{das}, $(X,D)$ is purely $F$-regular along $F_i$ and it follows that both $F_i$ are good $F$-pure centers for $(X,D)$, which contradicts Lemma \ref{lem:weak connectedness}. Hence we always have $r\lambda-2<A(1-r)$ and since $\lambda$ (resp. $r$) can be arbitrarily close to $(nv)^{-1}\vol(-K_X)$ (resp. $\alpha$), it follows immediately that $\vol(-K_{X})$ is bounded from above by a constant $M(n,A,v,\alpha)$ depending only on $n$, $A$, $v$ and $\alpha$.
\end{proof}

\begin{cor} \label{cor:mfs}
There exists a constant $M=M(n,\epsilon)$ depending only on $n$ and $\epsilon$ such that if $f:X\rightarrow Y$ is a Mori fiber space such that $X$ satisfies Assumptions \ref{ass:raylength} and \ref{ass:seshadri_bound}, then $\vol(-K_X)<M$.
\end{cor}

\begin{proof}
Let $F$ be the general fiber of $f$. By Lemma \ref{lem:F and Y}, $F\cong\mathbb{P}^{n-1}$ and $Y$ is a curve. By Corollary \ref{cor:fpt-Pn}, the existence of $M$ follows from Theorem \ref{thm:mfs} by taking any $v>n^{n-1}$ and $\alpha<1$. 
\end{proof}

In the remaining part of the section, we assume that $p>5$. We proceed to show that Assumption \ref{ass:seshadri_bound} implies Assumption \ref{ass:raylength} for Mori fiber spaces in dimension at most 3.

\begin{lem} \label{lem:raylength}
Let $(X,D)$ be a pair with $\dim X\le 3$ and $R$ a $(K_{X}+D)$-negative extremal ray. Assume that 
\begin{enumerate}
\item $R$ is generated by a curve; 
\item Every curve generating $R$ is not contained in the non-klt locus of $(X,D)$.
\end{enumerate}
Then $R$ is generated by a rational curve $C$ such that $0<-(K_{X}+D\cdot C)\le2\dim X$.
\end{lem}

\begin{proof}
%The proof is the same as that of \cite[Theorem 2.13]{jiang}, except that we use \cite[Theorem 1.2]{mmp-birkar} for the existence of log minimal model and \cite[Theorem 1.1]{mmp-bw} for cone theorem for klt pairs in dimension 3.
By \cite[Theorem 1.2]{mmp-birkar}, log minimal model exists for klt pairs in dimension 3 when $p>5$. Therefore by standard argument as in \cite[Corollary 1.4.4]{bchm}, there exists a birational morphism $\pi:X\rightarrow Y$, where $Y$ is $\bQ$-factorial, such that we can write $K_Y+\Gamma_1+\Gamma_2=\pi^*(K_X+D)$ where $(Y,\Gamma_1)$ is klt, $K_Y+\Gamma_1$ is $\pi$-nef and every component of $\Gamma_2$ has coefficient at least one. In particular, $\pi(\Supp(\Gamma_2))=\mathrm{Nklt(X,D)}$.

Since $\pi_*:\overline{NE}(Y)\to \overline{NE}(X)$ is surjective and $R$ is an extremal ray of $\overline{NE}(X)$, $\pi_*^{-1}(R)$ is an extremal face of $\overline{NE}(Y)$. By our first assumption, there exists a curve $C_0\subseteq X$ generating $R$. Take a curve $C'_0\subseteq Y$ such that $\pi(C'_0)=C_0$.
Then $(K_Y + \Gamma_1 + \Gamma_2)\cdot C'_0<0$. On the other hand, $C'_0\not \subseteq \Supp(\Gamma_2)$ by assumption (ii), hence $\Gamma_2\cdot C'_0\geq 0$ and we have $(K_Y + \Gamma_1)\cdot C'_0<0$. This implies that $\pi_*^{-1}(R)\cap \overline{NE}(Y)_{(K_Y+\Gamma_1)< 0}\neq \{0\}$. Since $\pi_*^{-1}(R)$ is an extremal face,  there exists
a $(K_Y+\Gamma_1)$-negative extremal ray $R'\subseteq \pi^{-1}_*(R)$. Since $K_Y+\Gamma_1$ is nef over $X$,  $\pi_*R'\neq \{0\}$, and hence $\pi_*R'=R$. By \cite[Theorem 1.1]{mmp-bw}, $R'$ is generated by a rational curve $C'$ such that $0<-(K_Y+\Gamma_1)\cdot C'\leq 2\dim X$. Hence $R$ is generated by the rational curve $C=\pi(C')$. By assumption (ii) again, we have $C'\not \subseteq \Supp(\Gamma_2)$ and $\Gamma_2\cdot C'\geq 0$, therefore $-(K_X+D)\cdot C\leq -(K_Y+\Gamma_1)\cdot C'\leq 2\dim X$.
\end{proof}

\begin{lem} \label{lem:3-fold-mfs}
Let $f:X\rightarrow Y$ be a Mori fiber space such that $\dim X\le 3$. Assume that $X$ satisfies Assumption \ref{ass:seshadri_bound}. Then $X$ also satisfies Assumption \ref{ass:raylength}.
\end{lem}

\begin{proof}
By Lemma \ref{lem:F and Y}, $\rho(X)=2$. Let $l$ be the class of a line in the general fiber $F$ of $f$ and let $R$ be the other extremal ray of $\overline{NE}(X)$. By Lemma \ref{lem:f-reg along F}, there exists $D\sim_\bQ -\frac{1}{1-\epsilon}K_X$ such that $(X,D)$ is klt along $F$. We may assume that $X$ is not weak Fano, otherwise there is nothing to verify. In particular, $(-K_X\cdot l)>0$ while $(-K_X\cdot R)<0$. Since $K_X+D\sim_\bQ -\frac{\epsilon}{1-\epsilon}K_X$ is $\bQ$-effective and has negative intersection with $R$, we see that $R$ is generated by a curve on $X$ by \cite[Proposition 5.5.2]{keel}. By construction, the non-klt locus of $(X,D)$ is contained in some special fibers of $f$, hence since $R$ has positive intersection with $F$, it satisfies all the assumptions of Lemma \ref{lem:raylength}. It follows that $R$ is generated by a curve $C$ such that
\[
0<\frac{\epsilon}{1-\epsilon}(K_X\cdot C)=-(K_{X}+D\cdot C)\le2\dim X\le 6
\]
and thus $X$ satisfies Assumption \ref{ass:raylength} by taking $A=\frac{6(1-\epsilon)}{\epsilon}$.
\end{proof}

\subsection{Picard number one case} \label{sec:rho=1}

Now we consider the case of terminal threefolds of Picard number one. Of course in this case Theorem \ref{thm:weakbdd_m} is just a special case of  the weak BAB conjecture in positive characteristic. Unfortunately this conjecture is still open even in dimension three, so we need a somewhat different approach. Similar to the fiber type case, our strategy is to prove an appropriate version of the Koll\'ar-Shokurov connectedness principle in positive characteristic and then, under the assumption that $X$ has large anticanonical volume, construct a boundary on $X$ that violates this principle. We start by setting up the framework.

Let $(X,B)$ be a pair and $x\in X$ a closed point such that $(X,B)$ is strongly $F$-regular at $x$. Let $D_1,\cdots,D_r$ be divisors on $X$ whose set theoretic intersection $\cap D_i$ equals $\{x\}$ in a neighbourhood of $x$. Let $\mathbf{D}=(D_1,\cdots,D_r)$ and let $\Delta(\mathbf{D})\subseteq\bR^r$ be the closure of the set of all $r$-tuples $(t_1,\cdots,t_r)\in\bQ^r_{\ge 0}$ such that $(X,B+\sum_{i=1}^r t_i D_i)$ is sharply $F$-pure at $x$. It can be viewed as an analog of the log canonical threshold polytope (see e.g. \cite{lct-polytope}) in positive characteristic. Clearly $\Delta(\mathbf{D})$ is convex. Let $\succ$ be the lexicographic ordering on $\bR^r$, namely, $t\succ t'$ if and only if $t\neq t'$ and the first non-zero entry of $t-t'$ is positive. We may then talk about the {\it dominant vertex} of $\Delta(\mathbf{D})$, defined to be the unique point $v=(a_1,\cdots,a_r)\in\Delta(\mathbf{D})$ such that for all $v'\in\Delta(\mathbf{D})$ we have $v\succeq v'$. Let $\Gamma(\mathbf{D})=\sum_{i=1}^r a_i D_i$. Any divisor $\Gamma$ of this form (i.e. there exists $\mathbf{D}$ as above such that $\Gamma=\Gamma(\mathbf{D})$) will be called an {\it $F$-pure combination} with an isolated center at $x$. Intuitively, one may view $(X,B+\Gamma)$ as an analog of a pair with an isolated log canonical center at $x$.

We can also define successive approximations of $\Gamma(\mathbf{D})$ as follows (c.f. \cite[\S 3]{bpf}). Let $(X,B)$ and $\mathbf{D}$ be as before and $e>0$ a positive integer such that $(p^e-1)(K_X+B)$ has integral coefficients, we define the the associated {\it $F$-threshold functions}  $t_i(e)\,(i=1,\cdots,r)$ inductively by taking $t_{l+1}(e)$ to be the largest integer $m\ge0$ such that the trace map
\[\mathrm{Tr}^{e}:F_{*}^{e}(\mathcal{O}_{X}((1-p^{e})(K_{X}+B)-\sum_{i=1}^l t_i(e)D_i-mD_{l+1}))\rightarrow\mathcal{O}_{X}\]
is locally surjective around $x$. It is then clear that $\frac{1}{p^e-1}(t_1(e),\cdots,t_r(e))\in\Delta(\mathbf{D})$ and their limit as $e\rightarrow\infty$ is exactly the dominant vertex of $\Delta(\mathbf{D})$. Let $W$ be the scheme-theoretic intersection of all the $D_i$, then by construction for all $j=1,\cdots,r$, 
\[\mathrm{Tr}^{e}:F_{*}^{e}(\mathcal{O}_{X}((1-p^{e})(K_{X}+B)-\sum_{i=1}^r t_i(e)D_i-D_j))\rightarrow\mathcal{O}_{X}
\]
is not surjective around $x$, thus
\begin{equation} \label{eq:not_surj}
\mathrm{Tr}^{e}(F_{*}^{e}(\mathcal{O}_{X}((1-p^{e})(K_{X}+B)-\sum_{i=1}^r t_i(e)D_i)\cdot \mathcal{I}_W))\subseteq\mathfrak{m}_x.
\end{equation}

We can now state the connectedness result we will use in this section:

\begin{lem} \label{lem:weak connectedness-2}
Let $(X,B)$ be a pair such that $X$ is projective and $\bQ$-factorial and there exists an integer $e>0$ such that $(p^e-1)B$ has integral coefficients. Let $x$, $y$ be general points on $X$ and let $\Gamma_x$ \emph{(}resp. $\Gamma_y$\emph{)} be an $F$-pure combination with an isolated center at $x$ \emph{(}resp. $y$\emph{)}. Then the divisor $-(K_X+B+\Gamma_x+\Gamma_y)$ is not ample.
\end{lem}

\begin{proof}
Since $x$, $y$ are general we may assume that they're smooth points. Let $\Gamma_x=\Gamma(\mathbf{D}_x)$ where $\mathbf{D}_x=(D_{1,x},\cdots,D_{r,x})$ and let $W_x=\cap_{i=1}^r D_{i,x}$. We have $\Supp(W_x)=\{x\}$. For sufficiently divisible integer $e>0$, let $t_{i,x}(e)$ be the $F$-threshold function associated to $\mathbf{D}_x$ at $x$ and let $\Gamma_x^{(e)}=\sum_{i=1}^r t_{i,x}(e)D_{i,x}$. Similarly we have corresponding objects indexed by $y$. Let $W=W_x\cup W_y$, then by (\ref{eq:not_surj}), we have
\[\mathrm{Tr}^{e}(F_{*}^{e}(\mathcal{O}_{X}((1-p^{e})(K_{X}+B)-\Gamma_x^{(e)}-\Gamma_y^{(e)})\cdot \mathcal{I}_W))\subseteq\mathfrak{m}_x\cdot \mathfrak{m}_y.
\] 
Hence for $L^{(e)}=\mathcal{O}_{X}((1-p^{e})(K_{X}+B)-\Gamma_x^{(e)}-\Gamma_y^{(e)})$ we have the following commutative diagram
\[\xymatrix{
0\ar[r] & F_{*}^{e}(\mathcal{I}_W\cdot L^{(e)})\ar[d]^{\mathrm{Tr}^{e}}\ar[r] & F_{*}^{e}L^{(e)}\ar[d]^{\mathrm{Tr}^{e}}\ar[r] & F_{*}^{e}(L^{(e)}|_W)\ar[d]^{\mathrm{Tr}^{e}}\ar[r] & 0 \\
0\ar[r] & \mathfrak{m}_x\cdot\mathfrak{m}_y\ar[r] & \cO_X\ar[r] & k_x\oplus k_y\ar[r] & 0.
}\]
By construction $\mathrm{Tr}^{e}:F_{*}^{e}L^{(e)}\rightarrow\cO_X$ is locally surjective around $x$ and $y$, thus the induced map $\mathrm{Tr}^{e}:F_{*}^{e}(L^{(e)}|_W)\rightarrow k_x\oplus k_y$ is also surjective. As $\dim W=0$, we get a surjection
\[H^0(W,F_{*}^{e}(L^{(e)}|_W))\twoheadrightarrow k_x\oplus k_y.
\]
On the other hand as $e$ goes to infinity $\frac{1}{p^e-1}\Gamma_x^{(e)}$ tends to $\Gamma_x$, thus if $-(K_X+B+\Gamma_x+\Gamma_y)$ is ample then by Fujita type vanishing we have 
\[H^1(X,F_{*}^{e}(\mathcal{I}_W\cdot L^{(e)}))=H^1(X,\mathcal{I}_W\cdot L^{(e)})=0,
\]
which implies that $H^0(X,F_{*}^{e}L^{(e)})\rightarrow H^0(W,F_{*}^{e}(L^{(e)}|_W))$ is also surjective. As in Lemma \ref{lem:weak connectedness}, we deduce that the natural restriction $H^0(X,\cO_X)\rightarrow k_x\oplus k_y$ is surjective, a contradiction. Hence $-(K_X+B+\Gamma_x+\Gamma_y)$ cannot be ample. 
\end{proof}

To apply Lemma \ref{lem:weak connectedness-2}, we need to find singular divisors whose associated $F$-pure combination has an isolated center. A key technical tool is provided by the next lemma, which allows us to construct new singular divisors out of existing ones. To state the result let us first recall a definition.

\begin{defn}[c.f. \cite{schwede_adjoint}] 
Let $(X,D)$ be a pair. The test module of $(X,D)$ is defined to be
\[\tau(K_X,D)=\sum_{e\ge0}\mathrm{Tr}^e_X(F_{*}^{e}(\omega_{X}(-\left\lceil p^{e}D\right\rceil ))) \subseteq \omega_X.
\]
\end{defn}

It is clear from our discussion in Section \ref{sec:test ideals} that over the smooth locus of $X$, the test module coincides with $\tau(X,D)\cdot\omega_X$ where $\tau(X,D)$ is the test ideal of $(X,D)$.

\begin{lem} \label{lem:global generation}
Let $(X,D)$ be a pair and $L$ a Weil divisor on $X$ such that $L-D$ is ample. Let $x\in X$ be a smooth point such that $\epsilon_F(L-D,x)>1$. Then the sheaf $\tau(K_X,D)\otimes\cO_X(L)$ is globally generated at $x$, where $\tau(K_X,D)$ is the test module of the pair $(X,D)$. 
\end{lem}

Here the notation $\epsilon_F(L,x)$ stands for the Frobenius-Seshadri constants \cite{F-seshadri,takumi} of the divisor $L$ at the smooth point $x$.

\begin{proof}
Let $e\gg 0$ be a sufficiently divisible integer and consider the following commutative diagram
\[\xymatrix@C=1em{
0\ar[r] & F_{*}^{e}(\mathfrak{m}_{x}^{[p^{e}]}\cdot\omega_{X}(-\left\lceil p^{e}D\right\rceil ))\ar[r]\ar[d]^{\mathrm{Tr}^{e}} & F_{*}^{e}(\omega_{X}(-\left\lceil p^{e}D\right\rceil ))\ar[r]\ar[d]^{\mathrm{Tr}^{e}} & F_{*}^{e}(\omega_{X}(-\left\lceil p^{e}D\right\rceil )\otimes\mathcal{O}_{X}/\mathfrak{m}_{x}^{[p^{e}]})\ar[r]\ar[d]^{\mathrm{Tr}^{e}} & 0\\
0\ar[r] & \mathfrak{m}_{x}\cdot\tau(K_X,D)\ar[r] & \tau(K_X,D)\ar[r] & \tau(K_X,D)\otimes k_{x}\ar[r] & 0.
}
\]
Since $x$ is a smooth point and $e\gg 0$, the trace map $\mathrm{Tr}^{e}:F_{*}^{e}(\omega_{X}(-\left\lceil p^{e}D\right\rceil ))\rightarrow\tau(X,D)\omega_{X}$ is locally surjective around $x$, hence after tensoring with $L$, we get another commutative diagram
\begin{equation} \label{eq:diagram}
\xymatrix{
F_{*}^{e}(\omega_{X}(p^{e}L-\left\lceil p^{e}D\right\rceil ))\ar[r]\ar[d] & F_{*}^{e}(\omega_{X}(p^{e}L-\left\lceil p^{e}D\right\rceil )\otimes\mathcal{O}_{X}/\mathfrak{m}_{x}^{[p^{e}]})\ar[d]^{\phi}\\
\tau(K_X,D)\otimes\cO_X(L)\ar[r] & \tau(K_X,D)\otimes\cO_X(L)\otimes k_{x}
}
\end{equation}
whose vertical maps are surjective around $x$. In particular, $\phi$ induces a surjection on global sections since both sheaves in question have zero-dimensional support. On the other hand, since $L-D$ is ample and $\epsilon_F(L-D,x)>1$, there exists $m\in\bZ_{\ge0}$ such that $(p^e-m)(L-D)$ is Cartier, $\omega_X(mL-\lceil mD \rceil )$ is globally generated and 
\[H^{0}(X,\cO_X((p^{e}-m)(L-D)))\rightarrow H^{0}(X,\cO_X((p^{e}-m)(L-D))\otimes\mathcal{O}_{X}/\mathfrak{m}_{x}^{[p^{e}]})
\]
is surjective. Tracing through the diagram, it follows that the two horizontal maps in (\ref{eq:diagram}) also induce surjection on global sections. In particular, $\tau(K_X,D)\otimes\cO_X(L)$ is globally generated at $x$.
\end{proof}

In the remaining part of this section, let $X$ be a $\bQ$-factorial terminal Fano threefold of Picard number 1 such that $\epsilon(-K_X)>2+\epsilon$. Suppose that $\vol(-K_X)$ can be arbitrarily large. Our goal is to derive a contradiction to Lemma \ref{lem:weak connectedness-2}. For this it suffices to find an $F$-pure combination $\Gamma$ with an isolated center at a very general point $x\in X$ such that $\Gamma\sim_\bQ -\lambda K_X$ for some $\lambda<\frac{1}{2}$. Roughly speaking, we will construct three divisors $D_1$, $D_2$, $D_3$ that are very singular at $x$ (so as to make $\Gamma$ small), and the main technical point is to cut down the dimension (at $x$) of their intersection (in general, singular divisors can be quite rigid and hard to deform).

The situation is very much like constructing isolated lc center in characteristic zero and it is always straightforward to come up with the first singular divisor. Let $x\in X$ be a very general point and let $N\in\bZ_{>0}$ be a sufficiently large constant that will be determined later. Suppose that $\vol(-K_X)>N^6$, then there exists an effective $\bQ$-divisor $D_1\sim_\bQ -K_X$ such that $\mult_x D_1>N^2$. Since $X$ has Picard number one, we may assume that $D_1$ is irreducible and write $D_1=t\Delta_1$ where $\Delta_1=\Supp(D_1)$. A priori the multiple $t$ can be large. Our first claim is that $t$ can be bounded in terms of $\epsilon$. More generally, we have the following (note that a variety with terminal singularities is smooth in codimension two).

\begin{lem} \label{lem:multiple bounded}
Let $X$ be a $\bQ$-factorial Fano variety of Picard number one and dimension $n$. Assume that $X$ is smooth in codimension two and $\epsilon(-K_X)>n-1+\epsilon$. Then there exists a constant $a=a(n,\epsilon)$ depending only on $n$ and $\epsilon$ such that for all prime divisor $D\subseteq X$ passing through a general point we have $-K_X\sim_\bQ tD$ for some $t<a$.
\end{lem}

\begin{proof}
Let $\nu:\tD\rightarrow D$ be the normalization of $D$ and let $\Delta$ be the conductor divisor on $\tD$. Since $X$ is smooth in codimension two, by adjunction we have $K_{\tD}+\Delta=\nu^*K_D=\nu^*(K_X+D)$. Let $x$ be a general point on $D$. By Theorem \ref{thm:Pn}, we have $\epsilon(-K_X-D,x)\le n+1$; on the other hand, for any curve $C\subseteq X$ through $x$ that's not contained in $D$, we have $\mult_x C \le (C\cdot D)$, thus if $t>n+2$, then $(-K_X-D\cdot C) \ge (t-1)\mult_x C > (n+1)\mult_x C$ for any such curve $C$. Therefore, by the definition of Seshadri constant it is not hard to see that
\begin{equation} \label{eq:compare_Seshadri}
\epsilon(-K_{\tD}-\Delta,x)=\epsilon(-K_X-D,x)>(1-\frac{1}{t})(n-1+\epsilon).
\end{equation}
So if $(1-\frac{1}{t})(n-1+\epsilon)>n-1$, then by Theorem \ref{thm:Pn} again we get $\tD\cong\bP^{n-1}$, $\Delta=0$ and $\epsilon(-K_{\tD}-\Delta,x)=n$. Substituting back to (\ref{eq:compare_Seshadri}) we see that $\epsilon(-K_X,x)>\epsilon(-K_X-D,x)=n$, which forces $X\cong\bP^n$ and contradicts $t>n+2$. It follows that we have either $t\le n+2$ or $(1-\frac{1}{t})(n-1+\epsilon)\le n-1$. In either case the existence of the constant $a(n,\epsilon)$ is clear.
\end{proof}

The following example shows that the assumptions in the lemma are necessary.

\begin{expl}
Let $X=\bP(1^2,d^{n-1})$ and let $H$ be the ample generator of $\mathrm{Cl}(X)$. Then for any point $x\in X$ there exists a divisor $D\sim H$ containing $x$ but $-K_X\sim (2+d(n-1))D$. Note that $\epsilon(-K_X)=n-1+\frac{2}{d}$.
\end{expl}

In the sequel, fix a constant $a=a(3,\epsilon)$ that satisfies the conclusion of the previous lemma. If $N\gg a$, then by the previous lemma the singularities of $D_1$ mainly come from the singularity of $\Delta_1$, which has codimension at least two. Our next claim is that if $N$ is sufficiently large, then we can find another divisor $D_2\sim_\bQ -K_X$, whose support does not contain $\Delta_1$ (so in particular, $\dim (D_1\cap D_2) \le 1$), such that $\mult_x D_2>bN$ for some (fixed) constant $b>0$. To construct the required divisor $D_2$ (and sometimes even $D_3$), we separate into two cases.

First suppose that $\mult_y D_1<N$ for all $y\neq x$ around $x$. Then by Lemma \ref{lem:mult and test ideal}, the test ideal $\tau(X,\frac{1}{N}D_1)\subseteq \mathfrak{m}_x^{N-2}$ and is trivial in a punctured neighbourhood of $x$. By \cite[Proposition 2.12]{F-seshadri}, \[\epsilon_F(-K_X,x)\ge \frac{1}{3}\epsilon(-K_X,x)>\frac{2}{3}.
\]
Thus if $N\gg 0$ and $L=-2K_X$ then $\epsilon_F(L-\frac{1}{N}D_1,x)>1$. Therefore by Lemma \ref{lem:global generation}, $\tau(K_X,\frac{1}{N}D_1)\otimes\cO_X(L)$ is globally generated at $x$. In particular, as $\tau(K_X,\frac{1}{N}D_1)\otimes\cO_X(L)=\tau(X,\frac{1}{N}D_1)\cO_X(-K_X)$ over the smooth locus of $X$, we get divisors $D_i\sim -K_X$ ($i=1,2,3$) such that locally $D_1\cap D_2 \cap D_3$ is supported at $x$ and $\mult_x D_i\ge N-2$. Let $\mathbf{D}=(D_1,D_2,D_3)$. By Lemma \ref{lem:mult and test ideal}, for any $t\in \Delta(\mathbf{D})$ we have $t_i\le \frac{3}{N-2}$, which gives $\Gamma(\mathbf{D})\le -\frac{6}{N-2}K_X<-\frac{1}{2}K_X$ as desired.

Next assume that there is a curve $C\subseteq X$ containing $x$ such that $\mult_y D_1\ge N$ for all $y\in C$ and by Lemma \ref{lem:multiple bounded}, $\mult_z D_1<a$ if $z\not\in C$ (all statements are local around $x$). Replacing $x$ by a general point of $C$ we may also assume that $C$ is smooth at $x$. Let $0<c<\min\{\frac{1}{2},\frac{1}{a}\}$. By Lemma \ref{lem:mult and test ideal} again, the test ideal $\tau(X,cD_1)\subseteq \mathfrak{m}_y^{\lfloor cN \rfloor -2}$ for all $y\in C$ and is trivial outside of $C$. Let $L=-3K_X$, then as in the previous case we have $\epsilon_F(L-cD_1,x)>\epsilon_F(-2K_X,x)>1$, thus $\tau(K_X,cD_1)\otimes\cO_X(-3K_X)$ is globally generated at $x$. In particular, we see that there exists a constant $b>0$ depending only on $\epsilon$ and a divisor $D_2\sim_\bQ -K_X$ whose support does not contain $\Supp(D_1)$ such that $\mult_y(D_i)>bN$ for all $y\in C\subseteq D_1\cap D_2$. 

Let $0<\delta<\frac{1}{2}$ be any rational number such that $\delta(2+\epsilon)>1$. In particular we get $\epsilon(-\delta K_X,x)>1$. By the definition of Seshadri constants, we have $(-\delta K_X\cdot C)>\mult_x C\ge 1$, hence there exists an effective divisor $D_3\sim_\bQ -K_X$ whose support doesn't contain $C$ such that $\mult_x(D_3|_C)>\delta^{-1}$. As before we may assume that $D_3$ is irreducible. Let $\mathbf{D}=(D_1,D_2,D_3)$. Our last claim is

\begin{lem}
$\Gamma(\mathbf{D})\sim_\bQ -\lambda K_X$ where $\lambda\le \frac{6}{N}+\delta$.
\end{lem}

\begin{proof}
Let $\Delta_i$ be the support of $D_i$ and write $D_i=b_i \Delta_i$ ($i=1,2,3$). Let $v=(a_1,a_2,a_3)$ be the dominant vertex of $\Delta(\mathbf{D})$ and $\nu_i(e)$ the $F$-threshold function at $x$ for the divisors $\Delta_i$ as before. By Lemma \ref{lem:mult and test ideal} it is clear that $a_1,a_2\le \frac{3}{N}$. We claim that $a_3\le \delta$. Let $W$ be the scheme-theoretic intersection of $\Delta_1$ and $\Delta_2$. Since $W$ is supported at $C$ (at least locally around $x$), there exists an integer $r>0$ such that $\cI_C^{[p^r]}\subseteq \cI_W$. For each $e\in \bZ_{>0}$, let $m(e)$ be the smallest integer such that $m(e)\cdot \mult_x(\Delta_3|_C)\ge p^e$. If $(f=0)$ is the local defining equation of $\Delta_3$, then we have
\[f^{m(e)}\subseteq \mathfrak{m}_x^{p^e} + \cI_C = \mathfrak{m}_x^{[p^e]} + \cI_C 
\]
where the second equality holds since $\mathfrak{m}_{C,x}^p=\mathfrak{m}_{C,x}^{[p]}$ on the smooth curve $C$. It follows that
\begin{equation} \label{eq:Delta_3}
f^{p^r m(e)}\subseteq (\mathfrak{m}_x^{[p^e]})^{[p^r]} + \cI_C^{[p^r]} \subseteq \mathfrak{m}_x^{[p^{e+r}]}+ \cI_W.
\end{equation}
It is clear that
\[\mathrm{Tr}^{e}(F_{*}^{e}(\mathcal{O}_{X}((1-p^{e})(K_{X}+B)-\sum_{i=1}^2 t_i(e)\Delta_i)\cdot \mathfrak{m}_x^{[p^e]}))\subseteq\mathfrak{m}_x,
\]
hence combining with (\ref{eq:not_surj}) and (\ref{eq:Delta_3}) we get
\[\mathrm{Tr}^{e+r}(F_{*}^{e+r}(\mathcal{O}_{X}((1-p^{e+r})(K_{X}+B)-\sum_{i=1}^2 t_i(e+r)\Delta_i-p^r m(e))\Delta_3))\subseteq\mathfrak{m}_x.
\]
Therefore, $\nu_3(e+r)\le p^r m(e)$ and as $a_3=\lim_{e\rightarrow\infty} \frac{\nu_3(e+r)}{b_3(p^{e+r}-1)}$, we deduce that 
\[a_3\le \frac{1}{\mult_x (D_3|_C)}<\delta
\]
as claimed. The lemma now follows as by our construction $D_i\sim_\bQ -K_X$ for each $i$.
\end{proof}

Summing up, we eventually have

\begin{thm} \label{thm:rho=1}
The set of $\bQ$-factorial terminal Fano threefolds $X$ of Picard number one such that $\epsilon(-K_X,x)>2+\epsilon$ for some $x$ is weakly bounded.
\end{thm}

\begin{proof}
Let $0<\delta<\frac{1}{2}$ be chosen as before, then by the previous lemma, as $N\gg 0$ there exists $\mathbf{D}=(D_1,D_2,D_3)$ such that $\cap D_i$ is locally supported at a very general point $x$ and $\Gamma(\mathbf{D})\sim_\bQ -\lambda K_X$ for some $\lambda<\frac{1}{2}$. This contradicts Lemma \ref{lem:weak connectedness-2}.
\end{proof}

Finally we finish the proof of Theorem \ref{thm:weakbdd_m} and its corollaries.

\begin{proof}[Proof of Theorem \ref{thm:weakbdd_m}]
This follows from Lemma \ref{lem:reduction}, Corollary \ref{cor:mfs}, Lemma \ref{lem:3-fold-mfs} and Theorem \ref{thm:rho=1}.
\end{proof}

\begin{proof}[Proof of Theorem \ref{thm:weakbdd}]
By \cite[Theorem 6.4]{demailly} we have $\epsilon(-K_X,x)=\epsilon_m(-K_X,x)$ for all smooth point $x\in X$, so the result follows immediately from Theorem \ref{thm:weakbdd_m}.
\end{proof}

\begin{proof}[Proof of Corollary \ref{cor:birbdd}]
By \cite[Theorem 1.1]{F-seshadri}, $|-2K_X|$ induces a birational map, so the corollary follows from Theorem \ref{thm:weakbdd} and \cite[Lemma 2.4.2]{hmx}.
\end{proof}

% \begin{acknowledgements}
% The author would like to thank his advisor J\'anos Koll\'ar for constant support, encouragement and numerous inspiring conversations. He also wishes to thank Takumi Murayama for several useful comments on an earlier draft of this paper and Yuchen Liu for helpful discussion.
% \end{acknowledgements}

\bibliography{ref}
\bibliographystyle{alpha}

\end{document}